\documentclass[12pt]{amsart}
\usepackage{graphicx}
\usepackage{tabularx}
\usepackage{mathtools}
\usepackage{subcaption}
\usepackage{float}

\usepackage{mathrsfs}     
\usepackage{helvet}         
\usepackage{courier}        
\usepackage{type1cm}      

\usepackage{todonotes}
\usepackage{color}
\usepackage{url}

\usepackage{geometry,calc,color}
\usepackage{amsmath,amssymb}
\usepackage{mathabx}

\usepackage{enumerate}
\usepackage{arydshln}
\usepackage{siunitx}
\usepackage{booktabs}
\usepackage{multirow}
\usepackage{tikz}
\usepackage{pgfplots}
\usepackage{enumitem}

\usepackage{esint}

\usepackage{appendix}
\pgfplotsset{compat=1.9}

\newtheorem{theorem}{Theorem}[section]
\newtheorem{corollary}[theorem]{Corollary}
\newtheorem{lemma}[theorem]{Lemma}
\newtheorem{proposition}[theorem]{Proposition}
\theoremstyle{definition}

\newtheorem{problem}[theorem]{Problem}
\newtheorem{assumption}[theorem]{Assumption}
\newtheorem{commento}[theorem]{Comment}
\theoremstyle{remark}
\newtheorem{remark}[theorem]{Remark}

\numberwithin{equation}{section}

\newcommand{\roundPrecision}{2}

\long\def\NOTE#1{} 
 
\allowdisplaybreaks

\usepackage[final]{changes}
\definechangesauthor[name=daniele, color=blue]{dp}
\definechangesauthor[name=monica, color=red]{mm}
\definechangesauthor[name=silvia, color=teal]{sb}

\newcommand{\Kt}{\widetilde K}

\renewcommand{\epsilon}{\varepsilon}
\renewcommand{\phi}{\varphi}
\renewcommand{\theta}{\vartheta}
\newcommand{\p}{k} 

\newcommand{\GDh}{\GD_h}

\newcommand{\fh}{F_h}

\newcommand{\tauD}{\tau}
\newcommand{\tauN}{\widehat\tau}

\newcommand{\lra}{\longrightarrow}

\newcommand{\bEdges}{\mathcal{E}_h^\partial}

\newcommand{\tE}{\widetilde E}

\newcommand{\DEdges}{\mathcal{E}_D}

\newcommand{\PinablaK}{\Pi^ {\nabla}_K}

\newcommand{\n}{{n}}

\newcommand{\scrG}{\mathscr{L}}

\newcommand{\h}{h}
\newcommand{\mh}{h}

\newcommand{\Poly}[1]{\mathbb{P}_{#1}}

\newcommand{\hu}{\widehat u}

\newcommand{\Th}{\mathcal{T}_\h}

\newcommand{\Gh}{\mathcal{G}_\h}

\newcommand{\x}{{\mathbf x}}

\newcommand{\Pinabla}{\Pi^{\nabla}}

\newcommand{\dn}[1]{\partial_{\n} #1 }

\newcommand{\Vh}{V_\h}
 
\newcommand{\Oh}{\Omega_{\mh}}

\newcommand{\ThD}{\Th^D}

\newcommand{\dnh}[1]{\partial_{\nh}#1}

\newcommand{\pvE}{\pv^{E}}

\newcommand{\nK}{\nu_K}

\newcommand{\VEM}{V_h}
\newcommand{\VEMK}{V^K_k}

\newcommand{\ah}{A_\h}

\newcommand{\mE}{E}

\newcommand{\KE}{K(E)}

\newcommand{\VE}{V_E}
\newcommand{\cVE}{\widecheck V_\mE}

\newcommand{\cv}{\widecheck v}

\newcommand{\cu}{\widecheck u_h}

\renewcommand{\hu}{\widehat u_h}

\newcommand{\bK}{{\partial K}}

\newcommand{\ThN}{\Th^N}

\renewcommand{\O}{\Omega}

\newcommand{\switch}{\chi}

\newcommand{\E}{\mathcal{E}}
\newcommand{\G}{\Gamma}

\newcommand{\dist}{\mathrm{d}}

\newcommand{\delD}{\delta_D}
\newcommand{\delN}{\delta_N}

\renewcommand{\phi}{\varphi}

\newcommand{\pu}{\pi_u}

\newcommand{\pv}{\pi_v}
\newcommand{\nh}{n_h}

\renewcommand{\Gh}{\Gamma_h}
\newcommand{\GD}{\Gamma^D}
\newcommand{\GN}{\Gamma^N}
\newcommand{\GhD}{\Gamma_h^D}
\newcommand{\GhN}{\Gamma_h^N}
\newcommand{\etax}{\eta_{\x}}

\renewcommand{\E}{\mathscr{E}}
\newcommand{\XN}{\mathcal{X}^N}

\newcommand{\DK}{\Delta_E}

\newcommand{\jump}[1]{[#1]}

\newcommand{\DNnodes}{\mathcal{N}_{DN}}

\newcommand{\Norm}[1]{|\!|\!| #1 |\!|\!|_{\flat,\Oh}}
\newcommand{\Normup}[1]{|\!|\!| #1 |\!|\!|_{\sharp,\Oh}}

\newcommand{\tK}{n_{\DK}}
\newcommand{\hE}{\widehat{\mathscr{E}}}
\newcommand{\taux}{\nu_{\x}}
\newcommand{\tx}{\tilde{\x}}

\newcommand{\domain}{G}
\newcommand{\Besicovich}{\mathscr{B}}
\newcommand{\Besk}{\Besicovich_k}
\newcommand{\Beskj}{\Besicovich^j_k}
\newcommand{\one}{{\mathbf 1}}

\newcommand{\linears}{L^K}
\newcommand{\bubbles}{B^K_k}
\newcommand{\node}{\zeta_i}
\newcommand{\Linear}{\mathcal{L}}
\newcommand{\Bubble}{\mathcal{B}}

\renewcommand{\tilde}[1]{\widetilde{#1}}

\begin{document}

\title[]{The high order virtual element method on approximate domains:  Neumann boundary conditions}%

\author[S. Bertoluzza]{Silvia Bertoluzza}
\address{IMATI ``E. Magenes'', CNR, Pavia (Italy)}%
\email{silvia.bertoluzza@imati.cnr.it}%

\author{Monica Montardini}
\address{Dipartimento di Matematica "F. Casorati", Universit\`a di Pavia}
\email{monica.montardini@unipv.it}

\author[D. Prada]{Daniele Prada}
\address{IMATI ``E. Magenes'', CNR, Pavia (Italy)}%
\email{daniele.prada@imati.cnr.it}%

\date{\today}
\thanks{This paper was funded by the MIUR Progetti di Ricerca di Rilevante Interesse Nazionale (PRIN) Bando 2020, grant 20204LN5N5, and Bando 2022 PNRR, grant P2022BH5CB (funded by the European Union - NextGeneration EU). \replaced[id=mm]{The first and second authors are member of the INdAM -- GNCS Research Group.}{The authors are member of the INdAM -- GNCS Research Group.} }%

\begin{abstract} 
In the framework of the virtual element method with shifted boundary type treatment of  curved geometries, we propose a novel approach to the treatment of Neumann boundary conditions that does not require the approximate aligning of the physical and numerical normal directions and is therefore well suited to handle domain approximations obtained by agglomeration from underlying fine structured grids. For such an approach we give conditions under which we are able to prove optimal error estimates for arbitrary orders of the virtual element discretization.
\end{abstract}
\maketitle

\section{Introduction}
In recent papers \cite{bertoluzza2024virtual,VEM_curvo,VEM_weakly} it has been shown that a combination of the Virtual Element Method (VEM) \added[id=mm]{\cite{3DVEM,basicVEM,hitchVEM}} with a boundary correction strategy \replaced[id=dp]{inspired by }{of} the {Shifted Boundary Method (SBM)}  \cite{SBM1}   allows to  obtain approximation of arbitrarily high order for Dirichlet problems set on domains with piecewise smooth boundaries, and solved on suitable polytopal  approximating domains. The VEM is, among the polytopal approximation methods that have\deleted[id=dp]{, in the last decade,} risen to the attention of the scientific community \added[id=dp]{in the last decade}, the closest on to being an extension of the classical finite element method to polytopal tessellations. As in the finite element method (FEM), the VEM relies, for second order elliptic equations, on a subspace of $H^1$ obtained by continuously gluing local spaces defined on each element of the tessellation. Unlike the FEM, the local basis functions are not available in closed form, and computations of the quantities needed for designing the discretization scheme are carried out approximately by resorting to suitable local projections onto polynomials. \added[id=dp]{Virtual elements have gathered an increasing interest for their flexibility in the definition of local element spaces that offers advantages in many complex problems, their robustness with respect to the shape of the elements, and the possibility to use arbitrary order of accuracy \added[id=mm]{\cite{3DVEM2}}. \added[id=mm]{Numerous works use VEM to deal with different kind of equations \cite{Antonietti_VEM_Stokes,Antonietti_VEM_Cahn,ABPV:minsurf,VEM_second_order,BEIRAODAVEIGA2021,navier-stokes2d,beirao_Navier_Stokes,Moraetal15,
perugia_Helmholtz,beirao_parab} and with  different formulations \cite{curved_Trefftz,de2016nonconforming,VEM_mixed,
BMPP_VEM_nonconforming_stab,MASCOTTO_Trefftz}. Indeed,} the range of applications is really vast: compressible and nearly incompressible materials \added[id=mm]{\cite{beirao_linear_elasticity,VEM_3D_elasticity}}, electromagnetics \added[id=mm]{\cite{BEIRAODAVEIGA2021}}, fluid dynamics \added[id=mm]{\cite{VEM_discrete_fracture}}, mesh adaptation to deal with singularities \added[id=mm]{\cite{beirao_hp_exponential}}, arbitrarily regular elements \added[id=mm]{\cite{da2014virtual}}, just to name a few. We refer to~\cite{vem_applications, beirao_acta2023} for further references}.

Initially designed under the assumption that the domain of definition of the continuous problem is a polygon or polyhedron, the VEM has been extended to domains with piecewise smooth curved boundaries by different approaches \cite{VEM_curved_beirao, VEM_curved_brezzi, prada2026virtual}. We focus here on boundary corrections techniques, such as the ones underlying the \replaced[id=dp]{SBM }{shifted boundary method (SBM)} \added[id=dp]{\cite{SBM1,SBM2,SBMmechanics, SBMfreesurface, SBMhyperbolic, SBMreduced}}. In such a method the domain of the problem is approximated by a polytope and the discrepancy between the two domains is taken into account when imposing Dirichlet boundary conditions,  by resorting to a Nitsche type strategy including a correction term that involves the evaluation on the physical boundary, by means of Taylor extrapolation, of the trial functions.  This strategy  has successfully been adapted to the VEM in \cite{VEM_curvo,VEM_weakly,hou2024high} and in \cite{bertoluzza2024virtual}. In particular, the latter paper focuses on the combination of VEM and boundary correction strategies for the solution of problems on domains with curved boundary approximated by polygonal tessellations with elements obtained as agglomerations of squares out of a uniform structured mesh. The use of polygonal elements allows to tune the ratio between the distance from the approximated to the physical boundary and the diameter of the element, which, in turn, allows to obtain optimal error bounds for the VEM of arbitrarily high order.

In the present paper we deal with the extension of such a method to the case of mixed Dirichlet/Neumann boundary conditions. The treatment of Neumann boundary \replaced[id=mm]{conditions}{condition} within the SBM and other boundary correction \added[id=dp]{methods} has been, until not long ago, an open problem. Different approaches were proposed in the literature. When the normal direction to the approximated boundary is close to the corresponding normal to the physical boundary, suitable boundary correction terms via Taylor extrapolation give satisfactory result, while otherwise the solution proposed in the literature is to resort, for the elements on the Neumann boundary, to mixed method \added[id=mm]{\cite{SBMmechanics}}, thus reducing the natural boundary condition to an essential boundary condition on the flux to be treated analogously to the Dirichlet boundary condition. A new approach has been recently proposed in \cite{collins2026gap}, that allows, in the order one case, to efficiently handle such boundary conditions even when the physical and numerical normals are not closely aligned, by cleverly recasting a numerical \replaced[id=mm]{quadrature }{quadrarure} for the ``gap'' domain between the two boundaries as a functional defined on the numerical boundary, thus remaining in the spirit of the shifted boundary method.  In this paper we  propose an alternative for which, in the   framework of the VEM, we are able to provide a full theoretical analysis for \deleted[id=dp]{methods of} arbitrarily high order \added[id=dp]{of accuracy}.

The paper is organized as follows. In Section \ref{sec:1} we present the new formulation and state our main result, namely Theorem \ref{thm:main}, that we prove in Section \ref{sec:proof}. Also in Section 2, \added[id=dp]{we present a novel stabilization strategy that is particularly well suited for the present framework} and recall the static condensation procedure originally introduced in \cite{bertoluzza2024virtual}, which plays an even greater role when dealing with Neumann boundary \replaced[id=mm]{conditions}{ condition}. In Section \ref{sec:numerical} we present numerical tests.

	\section{The virtual element method on domains with curved boundary} \label{sec:1}
	We let $\O\subset \mathbb{R}^2$ denote a piecewise smooth domain with (possibly curved) boundary $\G$, and outer unit normal $\n$, and we consider the mixed boundary value problem
	\begin{equation}\label{contpb}
		- \Delta u = f \ \text{ in }\O, \qquad u = g^D \ \text{ on }\GD, \qquad  \qquad \nabla u \cdot \n = g^N \ \text{ on }\GN,
	\end{equation}
  where $\GD$ and $\GN$ \deleted[id=mm]{ with $\widebar \GD \cup \widebar \GN = \G$ and $\GD \cap \GN = \emptyset$ } are the  Dirichlet and  Neumann portions of $\G$, respectively, \added[id=mm]{that satisfy    $\overline \GD \cup \overline \GN = \G$ and $\GD \cap \GN = \emptyset$}.
	
	\
	
We consider	a family \(\{\Oh\}\) of polygonal computational domains, depending on a meshsize parameter $h$, with, for all $h$, $\Oh\subseteq \O$. Each approximate domain is endowed with a polygonal tessellation $\Th$, which, for the sake of simplicity, we assume to be quasi uniform with mesh size $h$. We let $\Gh$ and $\nh$ respectively denote the boundary and \added[id=mm]{the} outer unit normal to the domain $\Oh$,
and we let $\bEdges$ \replaced[id=mm]{denote}{ denotes} the set of edges of $\Th$ lying on the boundary of $\Oh$. 
In the following it will be convenient to introduce the concept of  ``macro edges'',  defined as the connected components of either $\partial K \cap \partial K'$  or of $\partial K \cap \partial \Oh$, $K, K'\in \Th$. The common vertexes of two or more macro edges will be referred to as ``macro vertexes''. We make the following assumption on the tessellation.
	\begin{assumption}\label{assmp:shapereg} We assume that all the elements $K$ of $\Th$ satisfy the following shape regularity assumption
	\begin{enumerate}
		\item $K$ is star shaped with respect to all points of a ball of radius $\simeq h$.
		\item The number of macro edges of $K$ is bounded by a constant independent of $h$.
	\end{enumerate}	
\end{assumption}
	
	\
	
 	For all points $x\in \Gh$ on the boundary of the polygonal  computational domain we choose an extrapolation direction, represented by an outer unit vector $\sigma(x)$.  We let $\delta(x)\geq 0$ denote the smallest positive number such that $\widetilde x = x + \delta(x) \sigma(x) \in \G$. \added[id=sb]{For each boundary element $K$ we let $\delta_K = \max_{x \in \bK \cap \Gamma} \delta(x)$.}
	For each macro vertex $\x$ we let $\tx = \x + \sigma(\x)\delta(\x)$ denote the corresponding point on the physical boundary, and we let $\etax$ denote the segment $(\x,\tx)$.
	We make the following assumptions. 
	
	\begin{assumption}\label{ass:smoothness}
		The function $x \to \sigma(x)$ is continuous and its restriction to each boundary edge of the tessellation is smooth. 
	\end{assumption}
	
\added[id=sb]{	\begin{assumption}\label{ass:microedges}
For all boundary edge $e$ it holds that
\[
h_e = | e | \gtrsim \delta_K, \qquad K\text{ such that }e \subset \partial K.
\]
	\end{assumption}}
	
	\begin{assumption}
		For each boundary macro edge $E$, letting \[\tE = \left\{\widetilde x:\  \widetilde x = x + \delta(x)\sigma(x), \ x \in E\right\},\] we have either $\tE \subset \GD$ or $\tE \subset \GN$. 
	\end{assumption}
	Thanks to such an assumption, the computational boundary $\Gh$ can be split as $\Gh = \GhD \cup \GhN$,  the union being non overlapping,  where $\GhD$ and $\GhN$ are the union of those macro edges $E$ such that $\tE$ is, respectively, in $\GD$ or $\GN$. We let \[\delD = \max_{x\in \GhD} \delta(x)\qquad\text{ and }\qquad \delN = \max_{x\in \GhN} \delta(x).\]

\begin{figure}
	\includegraphics[width=7cm]{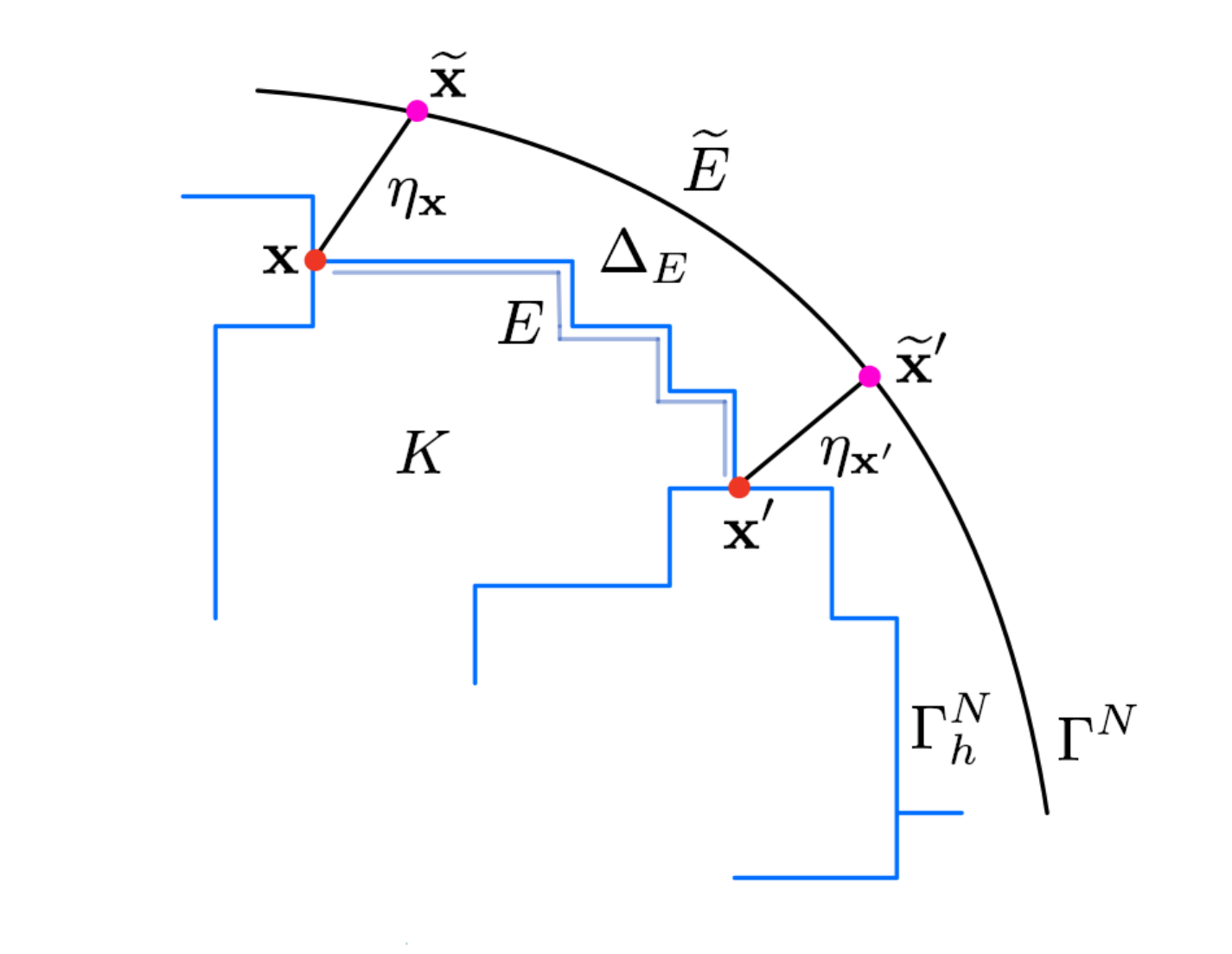}
	\includegraphics[width=7cm]{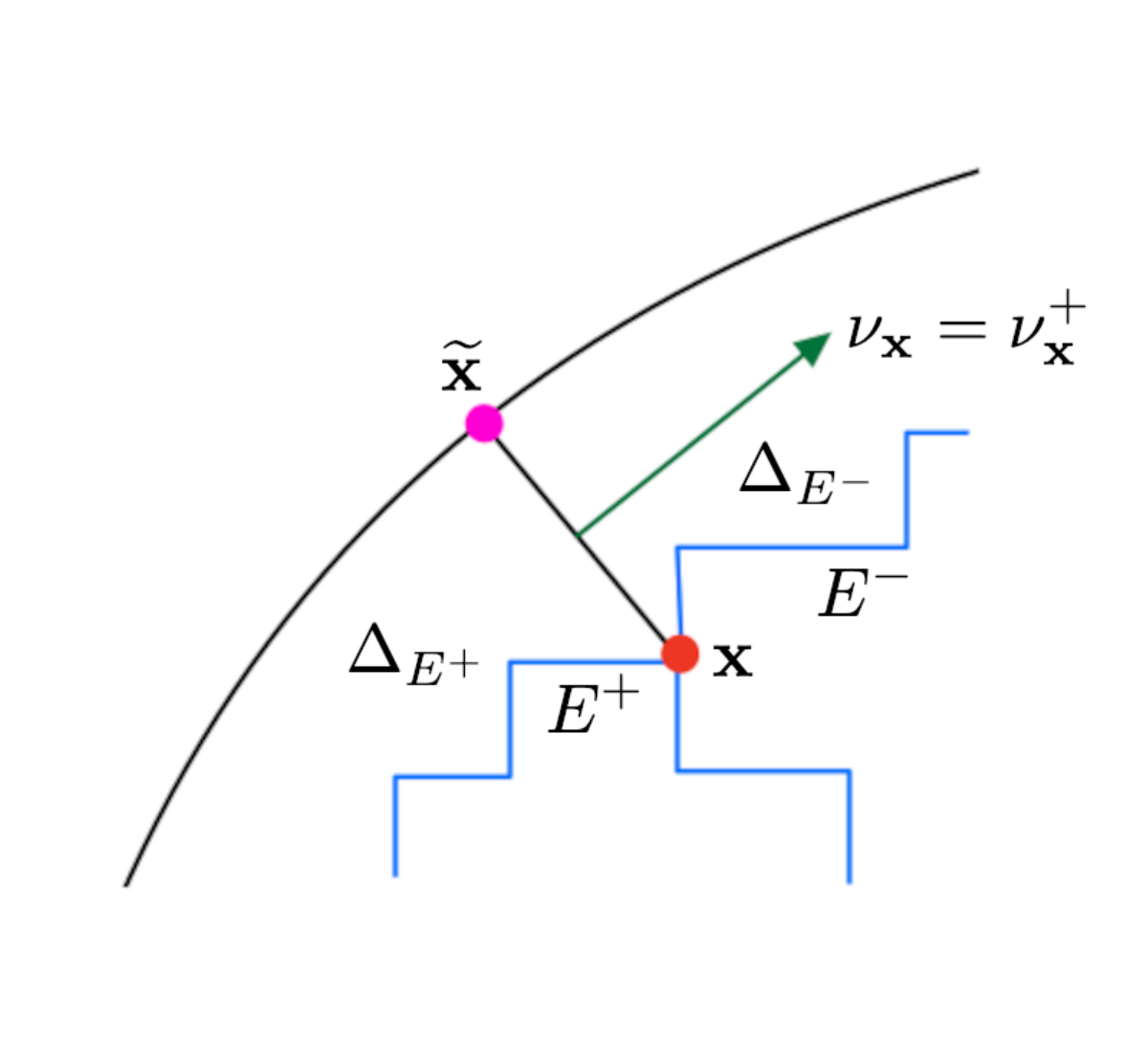}
	\caption{The different geometrical entities involved in the definition of the method, and the related notation.}\label{fig:geometry}
\end{figure}	

Corresponding to the splitting of 	the approximate boundary $\Gh$ into the union of macro edges $E$, the thin region between $\Gh$ and $\G$ can be split into subregions $\DK$, with $\DK$ defined as the bounded region with boundary $E \cup \tE \cup \etax \cup \eta_{\x'}$, $\x$ and $\x'$ being the macro vertexes of $E$ (see Figure \ref{fig:geometry}). 	For each macro vertex $\x$  we choose a  direction $\taux$ normal to $\etax$, and we mark the two macro edges sharing $\x$ as a macro vertex as $E^+$ and $E^-$, in such a way that $\taux^+ = \taux$ and $\taux^- = -\taux$ are pointing outwards of $\Delta_{E^+}$ and $\Delta_{E^-}$, respectively. To simplify the theoretical exposition, if $\x$ is one of the macro vertex that \replaced[id=mm]{separates}{ separate} $\GhD$ from $\GhN$, 	we take care to choose $E^+$ in the Neumann region.

\begin{remark}	Assumption \ref{ass:smoothness} can be relaxed. Indeed, continuity of $\sigma$ is only required on the Neumann portion of the boundary.
	On the Dirichlet boundary $\sigma$ can be taken piecewise constant on each edge of the polygonal approximate domain, as proposed in \cite{VEM_curvo}. We observe that  smoothness on Dirichlet edges is needed only for implementation, since Dirichlet boundary correction terms  requires numerical integration of quantities involving $\sigma$ and $\delta$. On the other hand, the implementation of the Neumann boundary terms only require to construct $\sigma$ at macro vertexes, while $\sigma$ at points interior to macro edges is only required for carrying out the theoretical analysis.
	\end{remark}

	\subsection{The virtual element method on $\Oh$} It is out of the scope of this paper to give a full presentation of the virtual element method, for which we refer to \cite{hitchVEM,basicVEM}. Here, we simply introduce the notation and definitions that we will need later on.	We will consider an order $k$ virtual element method. For the sake of simplicity we will focus on the plain definition of the VEM space, but the results that we will present also hold for the so called enhanced VEM space (see \cite{3DVEM}). For the polygonal elements $K \in \Th$ we let
	\[
	\VEMK = \left\{
	v \in C^0(K), v|_e \in \Poly{k}(e)\text{ for all edge $e$ of $K$}, \ -\Delta v \in \Poly{k-2}(K)
	\right\},
	\]
	where, for $D$ one- or two-dimensional domain, $\Poly{k}(D)$ denotes the  restriction to $D$ of the space of bi-variate polynomials of order up to $k$. As usual we conventionally write $\Poly{-1}(D) = \{0\}$.
	We define 
	\[
	V_h = \left\{ v \in H^1(\Oh),\ v|_K \in \VEMK\text{ for all $K \in \Th$}  \right\}.
	\]
	We let  $H^1(\Th)$ and $\Poly{k}(\Th)$ respectively denote the space of discontinuous piecewise $H^1$ functions and  the space of discontinuous piecewise polynomials of degree at most $k$ on the tessellation $\Th$. We let $| \cdot |_{1,\Th}$ and $\| \cdot \|_{1,\Th}$ respectively denote the broken $H^1$ semi norm and norm:
	\[
	| v |^2_{1,\Th} = \sum_{K} | u |_{1,K}^2, \qquad 	\| v \|^2_{1,\Th} = \sum_{K} \| u \|_{1,K}^2,
	\]
	where $| \cdot |_{1,K}$ and $\| \cdot \|_{1,K}$ denote the standard semi norm and norm for the Sobolev space $H^1(K)$. 
We let  $\PinablaK: H^1(K) \to \Poly{k}(K)$ and 
	$\Pinabla: H^1(\Th) \to \Poly{k}(\Th)$ be, as usual, respectively defined as
	\[\int_K \nabla \PinablaK v \cdot \nabla q = \int_K 
	\nabla v \cdot \nabla q\quad \text{for all }q \in \Poly{k}(K), \ \int_K \PinablaK v = \int_K v, \qquad
	\Pinabla(v)|_K = \PinablaK(v|_K).
	\]
  Moreover we let $\Pi^0_l: L^2(\Th) \to \Poly{l}(\Th)$ denote the $L^2(\Th)$ projection onto the space of discontinuous \replaced[id=mm]{piecewise}{ picewise} polynomials of degree at most $l$.
	It is out of the scope of this paper to detail  how the projectors $\PinablaK, \Pi^0_l$ are actually evaluated. We refer to \cite{hitchVEM} for an extensive discussion of the computability issue.
	
	As usual, to get control on the non polynomial component of the solution we will need to resort to stabilization, and to this aim we introduce  symmetric bilinear forms $s^K: \VEMK \times \VEMK \to \mathbb{R}$, $K \in \Th$, satisfying, for some constant $A^*(h) > \alpha_* > 0$
	\[
\alpha_*	| v_h |^2_{1,K} \leq s^K(v_h,v_h) \leq A^*(h)  | v_h |_{1,K}^2\qquad \text{for all }v_h \in \ker \PinablaK.
	\]
We assume the constant $\alpha_*$ to be independent of $h$, while, unlike what is usually assumed in the virtual element method analysis, it will be convenient to assume that $A^*(h)$ might depend on it (remark that, after possibly rescaling the stabilization bilinear form, this assumption is always satisfied). 	Different versions of $s^K$ have been proposed in the literature, {we refer to \cite{beirao_stab}} for a comprehensive discussion of various possibilities that can be found in the literature. Below we will also present a new option,  better suited to the present framework.

	\subsection{The discrete formulation} 	Given $v \in V_h$ we introduce the shorthand notation  $\pv = \Pinabla v$ to denote its polynomial component. 
	We treat the Dirichlet boundary conditions as proposed in \cite{VEM_weakly}. We then let 
	\[
	\E(u) = \sum_{j=0}^{k} \frac{\delta^j}{j!} \partial_\sigma^j u, \qquad \hE(u) = \sum_{j=0}^{\widehat k} \frac{\delta^j}{j!} \partial_\sigma^j u,
	\]
	denote two Taylor extrapolation operators in the $\sigma$ direction.
	Observe that for $x \in \Gh$ we have
	$\E(\pu) = \widetilde \pi_u$ where $\widetilde\pi_u(x) = \pu(\widetilde x)$. In essence, $\E(\pu)$ is a function defined on $\Gh$ obtained by ``transferring'' to $\Gh$, along the direction $\sigma$, the values of $\pu$ on $\G$, so that $\E(\pu)|_{\Gh}$ and $\pu|_\G$ carry the same information.
	
	\

To handle the Neumann boundary condition, we observe that $\pv$ can be naturally extended to a discontinuous piecewise polynomial  function defined on the whole domain $\O$, which, by abuse of notation, we will also call $\pv$,  by setting, for each boundary macro edge $E$,  \[\pv|_{\DK} = \pvE = \pv^{\KE}\quad \text{with $\KE$ such that } E \subset \partial \KE.\] 
	We let $\dnh{\pv}\in L^2(\G_h)$ and $\dn{\pv}\in  L^2(\G)$ be respectively defined macro edge by macro edge as
\[
\dnh \pv|_E = \nabla \pv^{\KE} \cdot \nh, \qquad \dn \pv|_{\tE} = \nabla \pv^{\KE} \cdot \n, \qquad \KE \text{ such that } E \subset \partial K.
\]

	\
	
	We then introduce the bilinear form $\ah  : \VEM \times \VEM \to \mathbb{R}$ defined as
		\begin{multline}
		\ah (u,v) =	\int_{\Th} \nabla \pu \cdot \nabla \pv  + \beta
		s_h(u-\pu,v-\pv) - \int_{\Gh} \dnh{\pu} \, v  \\+ \int_{\GhD}  \E(\pu) \dnh{\pv}   
		+ \gamma h^{-1} \int_{\GhD} \E(\pu) \hE (\pv)\\
		+ \int_{\GN} \dn{\pu}  \pv 
		+ \switch \sum_{\x  \in \XN}\int_{\etax}  [\nabla \pu\cdot \taux ]  \{ \pv \}.
	\end{multline}
	where \[[\nabla \pu \cdot \taux] = \nabla \pu^+ \cdot \taux^+ + \nabla \pu^- \cdot \taux^-,
	\qquad \{\pv\} = (\pv^+ + \pv^-)/2.
	\]
  Above, 
	the coefficients $\beta$ and $\gamma$ are stabilization parameters that one might need to adjust to ensure the stability of the method. The relevant values for the parameter $\switch$  are $1$ and $0$, its role being to switch on-off the last term in the definition of $a$, which is a correction term dealing with macro vertices which are off the physical boundary.  Letting
\[
  F_h(v) =  \int_{\Oh} f \Pi^0_{k-2} v + \int_{\GhD} \widetilde g^D (\dnh \pv + \gamma h^{-1} \hE(\pv) ) + \int_{\GN} g^N  \pv,
\]
we	consider the following discrete formulation.
		\begin{problem}\label{pb:discrete1}
		Find $u \in V_h$ such that for all $v \in V_h$ it holds
		\[
		\ah (u,v) = F_h(v).
		\]
	\end{problem}

%
%
%

The following theorem, which we will prove in Section \ref{sec:proof}, is the main theoretical result of the paper.
\begin{theorem}\label{thm:main}
There exists $\gamma_0$ and $\beta_0$ such that, if $\beta > \beta_0$ and $\gamma > \gamma_0$ the following holds: there exist $\tau_0$ depending on $\beta$ and $\gamma$,  such that, if the condition
\begin{equation}\label{condmesh}
\delta_D \leq \tau h, \qquad \delta_N \leq \tau h^{3/2}, 
\end{equation}
holds for $\tau < \tau_0$, then
Problem \ref{pb:discrete1} admits a unique solution $u_h \in \Vh$, and, if the solution $u$ of \eqref{contpb} satisfies $u \in H^{k+1}(\O)$, we have the following error estimate
\[
\Norm{u - u_h} \lesssim A^*(h) h^{k} | u |_{k+1,\O} + h^{-1/2} \delta^{k+1} \| u \|_{k+1,\infty,\O},
\]
where the discrete norm $\Norm{\cdot}$ is defined as
\begin{equation}\label{defnorm}
\Norm{v}^2 = | v |^2_{1,\Th} + h^{-1} \| v \|^2_{0,\GhD}.
\end{equation}
\end{theorem}

\begin{remark} A condition on the distance from the physical to the computational Neumann boundary  similar to \eqref{condmesh} also appears in   \cite{cheung_optimally_2019}.
\end{remark}

For convex domains in two dimensions it is always possible to construct polygonal domains and relative triangular meshes for which condition \eqref{condmesh} holds. Indeed, if we construct a triangular mesh whose boundary nodes lie on $\Gamma$, we have that $\delta \simeq h^2$, so that, for any given $\tau >0$ there exists $h_0$ such that, provided $h < h_0$, \eqref{condmesh} holds. This is generally not true for non convex domains, for which the requirement that $\Oh \subset \O$ implies that $\delta \simeq h$. Resorting to polygonal meshes, possibly (but not necessarily) obtained by agglomeration from triangular or quadrangular  meshes, allows for elements of diameter $h$ and edge-length $\delta$, with $h$ and $\delta$ satisfying \eqref{condmesh}. Unfortunately, in this case,  in order for the condition on the Neumann boundary, namely $\delta_N \leq \tau h^{3/2}$, to be satisfied as $h$ decreases, the concerned elements will have an increasing number of edges, thus violating one of the assumptions underlying the analysis of basically all the stabilizations available in the literature. In the next section we then present a novel  stabilization strategy that overcomes this limitation and that is particularly well suited for the present framework. For such a stabilization we will be able to prove that $A^*(h) \simeq \log(h)^2$. This will give us the following corollary.

\begin{corollary}
Under the assumptions of Theorem \ref{thm:main}, and assuming that the stabilization for elements adjacent to the Neuman boundary is defined according to \eqref{stabneumann}, if the solution of \eqref{contpb} satisfies $u \in H^{k+1}(\O)$  then we have that
	\[
	\Norm{u - u_h} \lesssim  h^k \log(h)^2 (| u |_{k+1,\O} + \| u  \|_{k+1,\infty,\O}).
	\]
\end{corollary}
\begin{remark}Theorem \ref{thm:main} requires the  approximate Neumann boundary to be closer to the continuous boundary than it is required for Dirichlet boundary conditions.
	It does not, however, require the normal to the approximate boundary to be an approximation to the normal to the continuous boundary.\end{remark}

\begin{remark}	 An alternative formulation of the Neumann boundary terms is obtained by \replaced[id=mm]{transferring}{ transfering} the integrals over $\G$ to integrals over $\Gh$ by using $\sigma$ to parametrize $\G$ by $\Gh$. 
	This alternative formulation is better suited  for handling Robin boundary conditions.
\end{remark}

\begin{remark} 
	Two other residual free stabilization terms might be added to the formulation, namely
	\[
	\int_{\GN} (\dn{\pu} - g) \dn {\pv} , \qquad \sum_{\x\in \XN} \int_{\etax} \jump{\pu} \jump{\pv }.
	\] 
	However, for the sake for simplicity, we do not consider such terms in our analysis, though it is possible to handle them with the same arguments we will use in the proof of Theorem \ref{thm:main}.
\end{remark}

\newcommand{\hh}{\widehat h}

\subsection{The stabilization term}\label{sec:robust-stabilization} To build a stabilization term robust with respect to the number of edges, we start by recalling \deleted[id=mm]{\cite{beirao_stab}} that, for $v_h \in \ker \PinablaK$, it holds that
\begin{equation}\label{stab2bd}
a(v_h,v_h) \simeq | v_h |^2_{1/2,\partial K}.
\end{equation}
\added[id=mm]{This is a result from \cite{beirao_stab}.}
We first consider a splitting of the $H^{1/2}$ norm into the sum of the contributions of the individual macro edges of $K$. The following lemma holds, \added[id=sb]{where $\hh$ denote the smallest edge length of $K$.}
\begin{lemma}\label{lem:breakonehalf} For all $v_h \in \VEMK$ with $\fint_{\partial K} v_h = 0$ it holds that
	\[  \log (\hh / h)^{-2}| v_h |_{1,\partial K}^2 \lesssim \sum_{E \text{ macro edge of }K} \left(| v_h |^2_{1/2,E} +  | \fint_{E} v_h | ^2\right)  \lesssim | v_h |_{1/2,\partial K}^2
	\]
\end{lemma}
Observe that, if   \replaced[id=sb]{$\hh\simeq\delta \simeq h^{3/2}$}{$\delta = h^{3/2}$}, then \replaced[id=sb]{$\log(\hh/h) \simeq \log(h^{1/2}) \simeq \log(h)$}{$\log(\delta/h) = \log(h^{1/2}) \simeq \log(h)$}.

\newcommand{\LE}{L^E}
\begin{proof} 
Using an inverse inequality and the equivalence of the norms of 
 $H^{1/2-\varepsilon}(E)$ and $H^{1/2-\varepsilon}(E)$ (see \cite{Bsub3field}) we can then write
	\begin{multline*}
		| v_h |_{1/2,\partial K} \lesssim \hh^{-\varepsilon}	| v_h |_{1/2-\varepsilon,\partial K} \lesssim  \hh^{-\varepsilon} | \sum_E \chi_{E} v_h |_{1/2-\varepsilon,\partial K}  \leq \sum_E \hh^{-\varepsilon} | v_h |_{H^{1/2-\varepsilon}_0(E)} \\
		\lesssim 
		\sum_E \hh^{-\varepsilon}\left(
		\frac {h^{\varepsilon}} \varepsilon \left(
		h^{-1/2} \| v_h - \fint_{E} v_h \|_{0,E} +  | v_h - \fint_E v_h |_{1/2,E}
		\right) + \frac {h^{\varepsilon}} {\sqrt{\varepsilon}} | \fint_{E} v_h |
		\right) \\
		\lesssim  
		\sum_E	 \frac {h^\varepsilon} {\hh^\varepsilon} \frac 1 \varepsilon
		\left(  | v_h  |_{1/2,E} +    | \fint_{E} v_h | \right).
	\end{multline*}
	Choosing $\varepsilon = 1/(\log(h/\hh))$ and exploiting the shape regularity of $K$ to pass to the squares, we get the first inequality, where the implicit constant in the inequality depends on the number of macro edges of $K$. 
  To prove the second inequality we observe that we have 
  \begin{multline*}
  	\sum_E |	\fint_E v |  \leq h^{-1} \sum_E \int_E | v | \lesssim h^{-1/2} \sum_E
  	\| v \|_{0,E} 
  	\lesssim h^{-1/2} ( h^{-1/2} \sum_E \| v \|_{0,E} + h^{1/2} \sum_E | v |_{1,E})\\
  	\lesssim h^{-1/2} (
  	h^{-1/2} \| v \|_{0,\bK} + h^{1/2} | v |_{1,\bK}) \lesssim | v |_{1,\bK}
  \end{multline*}
where we used Assumption \ref{assmp:shapereg} and a Poincar\'e inequality.  In view of \eqref{stab2bd}, this gives us the desired bound.
\end{proof}

\NOTE{
\begin{multline*}
\sum_E |	\fint_E v |  \leq h^{-1} \sum_E \int_E | v | \lesssim h^{-1/2} \sum_E
\| v \|_{0,E} 
\lesssim h^{-1/2} ( h^{-1/2} \sum_E \| v \|_{0,E} + h^{1/2} \sum_E | v |_{1,E})\\
 \lesssim h^{-1/2} (
h^{-1/2} \| v \|_{0,\bK} + h^{1/2} | v |_{1,\bK}) \lesssim | v |_{1,\bK}
\end{multline*}
}
Let us now, as a first step,  consider the case $k = 1$. We introduce the space
\[
\LE =V^K_1|_{E} = \left\{ v \in C^0(E): \ v|_e \in \Poly 1(e)  \ \text{ for all edge $e$ of $K$ with $e\subset E$}\right\}.
\]
Several options are available for defining a computable bilinear form $s^E_L: \LE \times \LE \to \mathbb{R}$ satisfying
\[
s_L^E(w,w) \simeq | w |_{1/2,E}^2 \quad \text{ for all }w\in \LE.
\]
We can, for instance, resort to a wavelet decomposition, or define the bilinear form at the algebraic level, as a matrix acting directly on the vector of nodal values of the $\LE$ elements, the matrix being defined as the square root of the 1-D stiffness matrix corresponding to the $H^1(E)$ semi scalar product.

\

Let  then \deleted[id=mm]{next} focus on the case $k \geq 2$.  We  let
\begin{gather}
\bubbles = \left\{v \in \VEMK|_{\bK} : v(\node) = 0\ \text{ for all node $\node$ of $K$} \right\},
\end{gather}
so that we can write 
\begin{equation}\label{splitting}
W = \VEMK|_{\bK} = \linears \oplus \bubbles.
\end{equation}
We have the following lemma.
\begin{lemma}
	For all $b \in \bubbles$ we have that
	\begin{equation}\label{normbubbles}
	| b |_{1/2,E}^2 \simeq \sum_{e\subset E} | b|_{1/2,e}^2,\end{equation}
		where the sum is taken over all edges $e$ of the polygon $K$ contained in the macro edge $E$. 
\end{lemma}
\begin{proof}
	We recall (see  \cite[Lemma 3.1]{bertoluzza2023localization}) that we have that 
	\[
	| b |^2_{1/2,E} \lesssim \sum_{e \subset E} \| b \|^2_{H^{1/2}_{00}(e)}.
	\]
	Moreover, as $| \cdot |_{1/2,e}$ is a norm on $\bubbles|_e$, by a scaling argument and by the equivalence of all norms on the  finite dimensional spaces $\Poly{k}(0,1)\cap H^1_0(0,1)$, we can see that
	\[
	\sum_{e \subset E} \| b  \|^2_{H^{1/2}_{00}(e)} \lesssim  \sum_{e \subset E} | b |^2_{1/2,e} \lesssim | b |^2_{1/2,E} 
	\]
	finally yielding \eqref{normbubbles}.
\end{proof}

We now let $\Linear: W \to \linears$ and $\Bubble: W \to \bubbles$ denote the linear operators mapping functions of $W$ into the two unique components of the splitting \eqref{splitting}:
\[
\text{for all }w\in W \qquad  w = \Linear(w) + \Bubble(w) \text { with } \Linear(w) \in \linears, \ \Bubble(w) \in \bubbles.
\] 
The splitting \eqref{splitting} is $H^{1/2}(E)/\mathbb{R}$ stable, as stated by the following lemma.
\begin{lemma}
	For all $w \in W$ with $\fint_E w = 0$ it holds that 
	\[
	| w |^2_{1/2,E} \simeq | \Linear(w) |^2_{1/2,E} + | \Bubble(w) |^2_{1/2,E}.
	\]
\end{lemma}
\begin{proof} We start by proving that for $w$ linear on $e$ and $b \in \Poly{k}(e)\cap H^1_0(e)$ we have that
	\[
	| w |_{1/2,e} \lesssim | w + b |_{1/2,e}.
	\]
	As the $H^{1/2}$ seminorm is scale invariant, it is sufficient to prove the bound on  the reference interval $I = (0,1)$. By direct computation, for $v \in \Poly 1 (I)$ we have, with $\tilde v = v - \fint v$
	\begin{multline*}
		| v  |_{1/2,I}  =| v(1) - v(0)|  = | [v+b](1) - [v+b](0) | = 
		| [v+b - \fint_I(v+b)](1) - [v+b - \fint_I(v+ b)](0) | 
		\\
		\leq | [v+b  - \fint_I(v+b)](1) | + | [v+b-\fint_I (v+b)](0) | \lesssim \| v+ b - \fint_I(v+b)\|_{\infty,I} \\
		\lesssim \| v+ b - \fint_I(v+b)\|_{1/2,I} \lesssim  | v+ b |_{1/2,I}.
	\end{multline*}
	By triangle inequality, this implies that for $b \in \Poly k(e)\cap H^1_0(e)$  and $w$ linear we have that
	\[
	| b |_{1/2,e} \leq | b + w |_{1/2,e} + | w |_{1/2,e} \lesssim | b + w |_{1/2,e}.
	\]
	Then, using \eqref{normbubbles}, we can write	\begin{equation}\label{stab1}
		| \Bubble(w) |^2_{1/2,E} \lesssim \sum_{e \subset E} | \Bubble(w) |^2_{1/2,e} \lesssim  \sum_{e \subset E} | \Bubble(w) + \Linear(w) |^2_{1/2,e}  =  \sum_{e \subset E} | w |^2_{1/2,e} \lesssim | w |_{1/2,E}^2,
	\end{equation}
	and, 
	by triangle inequality
	\begin{equation}\label{stab2}
		| \Linear(w) |^2_{1/2,E} \lesssim | \Linear (w) + \Bubble(w) |_{1/2,E}^2 + | \Bubble(w) |_{1/2,E} \lesssim | w |_{1/2,E}^2.
	\end{equation}
	The bounds \eqref{stab1} and \eqref{stab2} give us the upper bound. The lower bound is obtained by triangle inequality.
\end{proof}

To obtain the final form of the new \deleted[id=dp]{quasi robust} stabilization term, we finally observe that for all $b \in \bubbles$ it holds that
\[
| b |_{1/2,e} \simeq h_e^{-1/2} \| b \|_{0,e}.
\]
Indeed, the upper bound is an inverse inequality, while the lower bound is a Poincar\'e inequality. Then we have the final form of our new stabilization term, robust with respect to the number of edges:
\begin{equation}\label{stabneumann}
s^K (v,w) = \sum_E \left(s^E_1(\Linear(v),\Linear(w)) + \sum_{e\subset E} h^{-1}_e \int_e \Bubble(w)\Bubble(w) + \fint_E u \fint_E w \right).
\end{equation}
We point out that thanks to the splitting of the  stabilization term as the sum of independent  contributions from the different macro edges, the corresponding matrix will have a block diagonal structure, which will play a crucial role in the implementation, as we will see in the next section.

%
%
%
%
%

\subsection{Elimination of the ``lazy'' degrees of freedom} 
Constructing the domain $\Oh$ in such a way that the condition $\delta_N \leq \tau_0 h^{3/2}$ is satisfied might require to increase the number of boundary nodes per element as $h$ increases. This is the case, for instance, if the domain $\O$ is not convex, or if the domain $\Oh$ is to be built by agglomeration of elements of a structured triangular or quadrilateral grid. To avoid the resulting presence of larger and larger dense blocks in the stiffness matrix, we need to resort to the static condensation technique proposed in \cite{bertoluzza2024virtual}, which allows to eliminate all but few of the degrees of freedom ``living'' on the boundary \replaced[id=mm]{macro edges}{ macroedges}. Let us review the main ideas underlying such a procedure. For $E$ macro edge of the tessellation, we let 
\[
V_E = \left\{v \in V_h : v = 0\ \text{ on } \cup_K \partial K \setminus E, \ \int_K vp = 0, \forall K \in \Th,\ p \in \Poly{k-2}  \right\}.
\]
A necessary and sufficient condition for $\PinablaK v = 0$, $v$ being an element of $V_E$, is
\[
\int_E v p\cdot\nK = 0, \quad \forall p \in \Poly{k-1}^2. 
\]
We then split $V_E$ as
\[
V_E = \widehat V_E \oplus \widecheck V_E, \qquad \text{ with } \qquad \widecheck V_E =\left\{
v \in V_E\ :\  \int_E v p\cdot\nK = 0, \quad \forall p \in \Poly{k-1}^2
\right\},
\]
where $\widehat V_E$ is any subspace of $\VE$ such that the splitting holds. 
It is not difficult to ascertain that, for $v \in V_h$ arbitrary, it holds that
\[
	\ah (\cu,v) =	  \beta
	s_h(\cu,v-\pv).
\]
Then we can easily retrieve $\cu$ as a function of $\hu$ by setting $v = \cv \in \cVE$ as test function in Problem \ref{pb:discrete1}, resulting in the following local equation
	\[
\ah (\hu,\cv) + \beta s_h(\cu,\cv) = 0.
\]
We point out that $\cv$ does not directly depend on the data of the problem. This suggests  to define the space $\widehat V_E$ as
\[
\widehat V_E = \left\{v \in \VE : \ah(v,\cv) = 0\ \text{ for all } \cv \in \cVE\right\}.
\]
With such a definition, we can compute the solution to Problem \ref{pb:discrete1}, by directly solving the equation
\[
\ah (u,v) = \int_{\Oh} f v + \int_{\GhD} \widetilde g^D (\nabla \pv\cdot \nh + \gamma h^{-1} \hE(\pv) ) + \int_{\GN} g^N  \pv, \qquad \text{ for all } v \in \widehat V_E.
\]
We observe that, while the number of degrees of freedom for $V_E$, with $E$ Neumann macro edge, increases as $h$ goes to $0$, the number of degrees of freedom for $\widehat V_E$ stays uniformly bounded, and it is asymptotically comparable to the number of degrees of freedom that one would have if the macro edge $E$, which is the union of a number of small straight edges increasing as $1/\sqrt{h}$, was replaced with a single straight edge.

%
%
%

	\section{Proof of Theorem \ref{thm:main}}\label{sec:proof}
	
	We devote this section to proving Theorem \ref{thm:main}. To this aim we will establish continuity and coercivity of the bilinear form $\ah$, and we will provide an estimate of the consistency error. The error bound will follow by the second Strang Lemma. We let
		\[
		  \tauD = \max_{x\in \Gh} \frac{\delta(x)} h, \qquad \tauN = \max_{x \in \GhN} \frac{\delta(x)}{h^{3/2}}.
\]
Remark that the maximum in the definition of $\tauD$ is taken over the whole boundary $\Gh$, while for the definition of $\tauN$ the maximum is taken only over the Neumann portion of the boundary.

		\subsection{Preliminary bounds}
We start by recalling some known bounds and by proving some preliminary estimates that will be needed later on.  The following inverse inequalities for polynomials hold under our shape regularity assumptions.
	\begin{proposition}
		For all $p \in \Poly{k}(K)$ and for all 
		$m,j$ with  $0 \leq m \leq j$ it holds that
		\begin{gather}\label{inversebase}
			\| p \|_{j,K} \lesssim h^{m-j} \| p \|_{m,K},\\
			\label{tracePoly}
			\| p \|_{0,\partial K} \lesssim \,  h^{-1/2} \| p \|_{0,K}.
		\end{gather}
	\end{proposition}
	
	We next bound,  in term of their $L^2(K)$  norm,  the $L^2(\DK)$ and $L^2(\etax)$ norm of polynomials, $E\subset \partial K$ and $\x$  being a boundary macro edge and a boundary macro vertex of $K\in \Th$, respectively.
	\begin{proposition}\label{prop:2.2}
		For all $p$ in $\Poly{k}$ the following bounds hold, the  implicit constant in the inequalities only depending on the polynomial degree $k$ and on the shape regularity of the element $K$. For $E\subset \partial\Oh$ boundary macro edge and $\x \in \partial\Oh$ boundary macro vertex of $K \in \Th$ it holds that
		\begin{gather}\label{inverseDK}
			\| p \|^2_{0,\DK}  \lesssim \frac \delta h \| p \|^2_{0,K},
			\qquad 
			\| p \|^2_{0,\etax}
			\lesssim \frac \delta {h^2}\| p \|^2_{0,K}. 
	\end{gather}
	
\end{proposition}

\begin{proof} We exploit the inverse inequality of the form
\[	\| p \|_{0,\infty,B_K} \lesssim h^{-1} \| p \|_{0,K},\] where $B_K$ is the ball with center coinciding with the baricenter of $K$ and  diameter $C h$, with $C$ sufficiently big so $\Delta_E \subseteq B_K$. We can write
\begin{equation*}
	\| p \|^2_{0,\DK} \leq | \DK | \| p \|^2_{0,\infty,\DK} \lesssim \delta h \| p \|^2_{0,\infty,K} \lesssim \frac \delta h \| p \|_{0,K}^2.
\end{equation*}
To prove the second bound we write
\[
\| p \|^2_{0,\etax} \lesssim \delta \| p \|^2_{0,\infty,\etax} \lesssim 
\delta \|p \|^2_{0,\infty,B_K} 
\lesssim \frac \delta {h^2}\| p \|^2_{0,K}.
\]
\end{proof}
Applying Proposition \ref{prop:2.2} to the gradient of $p$, which is also a polynomial we get that $| p |_{1,\DK} \lesssim (\delta/h) | p |_{1,K}$, and then  space interpolation yields the following corollary.
\begin{corollary}Under the assumptions of Proposition \ref{prop:2.2} for all $0 \leq s \leq 1$ it holds that
	\[
	\| p \|_{s,\DK} \lesssim \frac \delta h \| p \|_{s,K}.
	\]
\end{corollary}

The last bound that we will need is a bound of the $H^s_0(\DK)$ norm, $s<1/2$,  of a polynomial, in terms of its $L^2(K)$ norm, where, once again, $E$ is a macro edge of $K\in \Th$. Such a bound is a corollary of the following Lemma, the proof of which is the object of Appendix \ref{sec:embedding}.

\begin{lemma}\label{embedding}
Let $\domain\subseteq \mathbb{R}^2$ be a bounded domain with  piecewise smooth boundary $\partial\domain$, and let $\rho_\domain$ denote the radius of the largest inscribed ball. Let $f \in  H^s(\domain) \cap L^\infty(\domain)$. Then for $s \in ]0,1/2[$  it holds that 
\[
\| f \|_{H^s_0(\domain)} \lesssim  \| f \|_{H^{s}(\domain)}  + \rho_\domain^{1/2-s}
\sqrt{| \partial\domain |}\,   \frac 1 {\sqrt{(1-2s)}}  \| f \|_{L^\infty(\domain)}.  \]
\end{lemma}

\begin{corollary}\label{cor:embedding}
For all polynomial $p \in \Poly{k}$ and all $s > \frac 1 4$, it holds that
\[
\| p \|_{H^s_0(\DK)} \lesssim \left(
\frac \delta h 
\right)^{1/2-s}
\frac {h^{-s} }{\sqrt{(1-2s)}} \| p \|_{0,K}.
\]
\end{corollary}

\NOTE{
\begin{proof}[Proof of Corollary]
	\begin{multline}
		\| p \|_{H^s_0(\DK)} \lesssim | p |_{H^{s}(\DK)}  + h^{1/2}
		\delta^{1/2-s} \frac 1 {\sqrt{s(1-2s)}}  \| p \|_{L^\infty(\DK)} \\
		\lesssim \| p \|_{0,\DK}^{1-s} | p |_{1,\DK}^{s} + h^{1/2}
		\delta^{1/2-s} \frac 1 {\sqrt{s(1-2s)}} \| p \|_{L^\infty(\DK)} \\
		\lesssim \left(
		\frac \delta h 
		\right)^{1/2}  \| p \|_{0,K}^{1-s} | p |_{1,K}^{s} +  h^{1/2}
		\delta^{1/2-s} \frac 1 {\sqrt{s(1-2s)}} \| p \|_{L^\infty(K)}
		\\
		\lesssim 
		\left(
		\frac \delta h 
		\right)^{1/2}   h^{-s} | p |_{0,K} +  h^{1/2}
		\delta^{1/2-s} \frac 1 {\sqrt{s(1-2s)}} h^{-1} \| p \|_{0,K}\\
		=
		\left(
		\frac \delta h 
		\right)^{1/2}   h^{-s} | p |_{0,K} +  
		\delta^{1/2-s}  h^{s -1/2}    h^{-s}   \frac 1 {\sqrt{s(1-2s)}}  \| p \|_{0,K} \\
		\lesssim \left(
		\frac \delta h 
		\right)^{1/2-s}
		\frac {h^{-s} }{\sqrt{s(1-2s)}} \| p \|_{0,K}.
	\end{multline}
\end{proof}
}

	\subsection{Continuity} We start the analysis of the method by proving a continuity bound for the bilinear form $\ah $ with respect to the norm $\Norm\cdot$ (which is defined in \eqref{defnorm}). We claim that, letting
	\[
	\Normup v^2 = | \pv |_{1,\Th}^2 + A^*(h) | u - \pv |_{1,\Th}^2 + h^{-1} \| v \|_{0,\GD}^2	\]
	 for all $u, v \in H^1(\Th)$ it holds that
	\begin{equation}\label{continuity}
		\ah(u,v) \lesssim \Normup u \Normup v
	\end{equation}

	To prove \eqref{continuity}, we begin by observing that the following Poincar\'e type inequality holds for the (discontinuous) polynomial component $\pu = \Pinabla u$  of functions $u \in H^1(\Oh)$.
	\begin{proposition}[Discrete Poincar\'e inequality]\label{prop:poincare}
		For all $u \in H^1(\Oh)$	it holds
		\begin{equation}\label{polypoincare}
		\| \pu \|^2_{0,\Oh} \lesssim 	 | \pu - u |_{1,\Th}^2 + | \pu |_{1,\Th}^2 
		+ \sum_{e \subset \GD} \| \pu \|_{0,e}^2.
		\end{equation}
	\end{proposition}
	
	\begin{proof} We start by observing that a Poincar\'e bound holds, of the form
		\begin{multline}\label{poinc1}
		\| u \|^2_{0,\Oh}  \lesssim | u |^2_{1,\Oh} + |  \fint_{\GDh} 	u |^2 \leq | u |^2_{1,\Oh} + \fint_{\GDh} |
		u |^2 \lesssim | u |^2_{1,\Oh} + \sum_{e\in \DEdges}  \| u \|_{0,e}^2\\
		\lesssim 
 | \pu |^2_{1,\Oh} + | u - \pu |^2_{1,\Oh}  	+ \sum_{e \subset \GD}  \| \pu \|_{0,e}^2 + 
 \sum_{e \subset \GD}  \| \pu - u  \|_{0,e}^2\\
\lesssim   
 | \pu - u |_{1,\Th}^2 + | \pu |_{1,\Th}^2 
 + \sum_{e \subset \GD} \| \pu \|_{0,e}^2 
		\end{multline}
	where the last bound is obtained by combining a trace inequality with an Aubin-Nitsche argument allowing to bound the $L^2(K)$ norm of $u - \pu$ with $h$ times its $H^1(K)$ semi norm.
		Now conclude by observing 
		\begin{gather*}
			\| \pu \|^2_{0,\Th} \lesssim \| \pu - u \|^2_{0,\Th} + \| u \|^2_{0,\Oh}   \lesssim h | \pu - u |_{1,\Th}^2 +\| u \|^2_{0,\Oh},
		\end{gather*}
	which, combined with \eqref{poinc1}, allows to obtain \eqref{polypoincare}.
	\end{proof}

	Using the same arguments as \cite[Lemma 2.2]{VEM_curvo}, we can write, for all $u, v \in H^1(\Oh)$,
	\begin{multline*}
		\ah (u,v) \lesssim \Normup{u} \Normup{v}	 	- \int_{\Gh} \dnh{\pu}\,(v - \pv ) \\-\underbrace{ \int_{\Gh} \dnh{\pu}\, \pv	}_{\star}+ \int_{\GN} \dn{\pu} \, \pv + \switch \sum_{\x  \in \XN} \int_{\etax}  [\nabla \pu\cdot \taux ]  \{ \pv \},
	\end{multline*}
where we added and subtracted $\star$.	To bound the three last terms on the right hand side we resort to the following lemma, which deals with the case $\switch = 1$.
	
	\begin{lemma}\label{lem:boundI} For $u,v \in V_h$ it holds that
		\begin{multline}\label{contpezzoNeumann}
			- \int_{\GhN} \dnh{\pu} 
			\, \pv	+ \int_{\GN} \dn{\pu}\,  \pv + \sum_{\x  \in \XN} \int_{\etax}  [\nabla \pu\cdot \taux ]  \{ \pv \}
			\\ \lesssim \tau  | \pu |_{1,\Th} | \pv |_{1,\Th} +  \tauN | \pu |_{1,\Th} \| \pv \|_{0,\Oh}
			.
		\end{multline}
	\end{lemma}
	
	\begin{proof}
We start by observing that
		\begin{multline*}
			\sum_{\x  \in \XN} \int_{\etax}  [\nabla \pu\cdot \taux ]  \{ \pv \} = \sum_{\x  \in \XN} \int_{\etax} 
			(\nabla \pu^+ \cdot \taux^+ + \nabla\pu^- \cdot \taux^- )(\pv^+ + \pv^-)/2 \\=  \sum_{\x  \in \XN}  \left(
			\int_{\etax} \nabla \pu^+ \cdot \taux^+ \pv^+ 
			+ \int_{\etax} \nabla \pu^- \cdot \taux^- \pv^- 
		-  \int_{\etax} \nabla \{\pu \} \cdot \taux \jump{\pv}  	\right),
		\end{multline*}
		where $\pi_u^+$ (resp. $\pi_u^-$) are the restrictions of the discontinuous polynomial function $\pi_u$ to, respectively, $\Delta_{E^+}$ and $\Delta_{E^-}$ (see Figure \ref{fig:geometry}).

		\NOTE{
		\begin{multline*}		 
			(\nabla \pu^+ \cdot \taux^+ + \nabla\pu^- \cdot \taux^- )(\pv^+ + \pv^-)/2 \\
		= 	\nabla \pu^+ \cdot \taux^+ \pv^+  + \nabla\pu^- \cdot \taux^- \pv^+/2  
		+ 	\nabla \pu^+ \cdot \taux^+ \pv^-/2 + \nabla\pu^- \cdot \taux^- \pv^-
		- \nabla \pu^+ \cdot \taux^+ \pv^+/2  
-		\nabla\pu^- \cdot \taux^- \pv^-/2 \\
=
\nabla \pu^+ \cdot \taux^+ \pv^+  
 + \nabla\pu^- \cdot \taux^- \pv^-
+ \nabla\pu^- \cdot \taux^- \pv^+/2  
+ 	\nabla \pu^+ \cdot \taux^+ \pv^-/2
- \nabla \pu^+ \cdot \taux^+ \pv^+/2  
-		\nabla\pu^- \cdot \taux^- \pv^-/2 
		\end{multline*}	
We can write
\begin{multline*}
-	 \nabla\pu^- \cdot \taux \pv^+/2  
	+ 	\nabla \pu^+ \cdot \taux^+ \pv^-/2
	- \nabla \pu^+ \cdot \taux^+ \pv^+/2  
	+	\nabla\pu^- \cdot \taux \pv^-/2 \\
	= 
	+ 	\nabla \pu^+ \cdot \taux (\pv^- - \pv^+)/2
	+	\nabla\pu^- \cdot \taux (\pv^- - \pv^+)/2 \\
	=\nabla\left(\frac{\pu^- + \pu^+}2\right) \cdot \taux (\pv^- - \pv^+)
\end{multline*}
	}
		We also observe that, letting $\x$ and $\x'$ be the two macro vertexes of $E$, and assuming, to fix the ideas, that both $\taux^+$ and $\nu_{\x'}^+$ point outwards from $\DK$, we have
		\begin{multline*}
		-\int_E \dnh{\pu} \pv + \int_{\widetilde E} \dn{\pu} \pv +  \int_{\etax} \nabla \pu^+ \cdot \taux^+ \pv^+ + \int_{\eta_{\x'}}  \nabla \pu^+ \cdot \nu^+_{\x'} \pv^+  \\
		= 
		 \int_{\DK} \Delta \pu \pv + \int_{\DK} \nabla \pu \cdot \nabla \pv.	\end{multline*}
		
		Then we have that
		\begin{multline}\label{boundI1}
			- \int_{\GhN} \dnh{\pu} \,
			\pv	+ \int_{\GN} \dn{\pu}\,  \pv + \sum_{\x  \in \XN} \int_{\etax}  [\nabla \pu\cdot \taux ]  \{ \pv \}
		\\	=\underbrace{ \sum_{E\subseteq \GhN} \left(\int_{\DK} \nabla \pu \cdot \nabla \pv + \int_{\DK} \Delta \pu \pv \right)}_{B_1}
			- \underbrace{\sum_{\x \in \XN}  \int_{\etax} \nabla \{\pu \} \cdot \taux \jump{\pv} }_{B_2}
			+ \underbrace{\sum_{\x \in \DNnodes} \int_{\etax }\nabla\pu^- \cdot \taux^- \pu^-  }_{B_3}
		\end{multline}
	where $\DNnodes$ is the set of nodes shared by a Dirichlet and a Neumann macro edge.
	
	\
	
				We bound the three terms separately, starting with $B_1$. 
We  observe that the linear operator $\scrG: L^2(\DK) \to \mathbb{R}$ defined as
			\begin{equation*}
			\scrG(v) = \int_{\DK} \Delta \pu v 
			\end{equation*}
			satisfies
			\begin{gather*}
			| \scrG(v) | \leq \| \Delta \pu \|_{0,\DK} \| v \|_{0,\DK}, \ \forall v \in L^2(\DK), \\	| \scrG(v) | \leq \| \Delta \pu \|_{-1,\DK} \| v \|_{1,\DK} \leq \| \nabla \pu \|_{0,\DK}  \| v \|_{1,\DK} \forall v \in H^1_0(\DK).
			\end{gather*}
		\NOTE{
	The constants are one with the unscaled norms: indeed we have that 
	\[
	\int_{\DK} \Delta u v = \int_{\DK} \nabla u \cdot \nabla v \leq | u |_{1,\DK} | v |_{1,\DK} \leq | u |_{1,\DK} | v |_{1,\DK}.
	\]	
	The definition of the $H^s_0$ norm is the one compatible with interpolation, as it is defined by extension.
	}
		
			By interpolation this implies that for all $v \in H^s_0(\DK) = [L^2(\DK),H^1_0(\DK)]_s$, it holds that
			\[
			|  \scrG(v) | \leq \| \nabla \pu \|^s_{0,\DK} \| \Delta \pu \|^{1-s}_{0,\DK} \| v \|_{H^s_0(\DK)}.
			\]
			
			\
			
 We can now combine this bound with Corollary \ref{cor:embedding}, and, recalling that $\pv \in H^s_0(\DK)$ for all $s<1/2$, we can write
			\begin{multline*}
			 \int_{\DK} \Delta \pu \pv = \scrG(\pv) \lesssim \| \nabla \pu \|^s_{0,\DK} \| \Delta \pu \|^{1-s}_{0,\DK} \| \pv \|_{H^s_0(\DK)}
	\\	\lesssim  \left(
	\frac \delta h 
	\right)^{1/2}  \| \nabla \pu \|^s_{0,\DK} 
	\| \Delta \pu \|^{1-s}_{0,K} \frac{1}{\sqrt{\varepsilon}} \left(
		\frac \delta h 
		\right)^{\varepsilon} 
	 {h^{-s}  } \| \pv \|_{0,K} \\
		\lesssim 
		 \left(
		\frac \delta h 
		\right)^{1/2}  | \pu |_{1,K} h^{-1} 
	\frac{1}{\sqrt{\varepsilon}} \left(
\frac \delta h  
\right)^{\varepsilon} \| \pv \|_{0,K} 
\\= \frac 1{\sqrt{\varepsilon}} \left(
\frac \delta h  
\right)^{\varepsilon}  	 \left(
\frac \delta {h^{3/2}} 
\right)^{1/2}  | \pu |_{1,K}  \| \pv \|_{0,K},
	\end{multline*}
with $\varepsilon = 1/2 - s$.
 I then choose $\varepsilon = 1/| \log(\delta/h) |$, which yields
	\[
	\frac 1{\sqrt{\varepsilon}} \left(
	\frac \delta h  
	\right)^{\varepsilon}  \lesssim \sqrt{| \log(\delta/h) |}  \lesssim 1.
	\]
Then
	\begin{equation}
		 \int_{\DK} \Delta \pu \pv \lesssim  \left(
		 \frac \delta {h^{3/2}} 
		 \right)^{1/2}  | \pu |_{1,K}  \| \pv \|_{0,K},
	\end{equation}
and this gives us
	\begin{equation*}
		\int_{\DK}  \nabla {\pu^K} \cdot \nabla {\pv^K} + \int_{\DK} \Delta \pu^K \pv^K
		\lesssim \	\frac{\delta}{h} |  \pu |_{1,K} |  \pv |_{1,K} + \frac \delta {h^{3/2}} |  \pu |_{1,K}  \| \pv \|_{0,K}.
	\end{equation*}
		Adding up the contributions from the different Neumann macro edges we then can write
		\begin{gather*}
			B_1
			\lesssim \frac \delta h | \pu |_{1,\ThN} |  \pv |_{1,\ThN} +  \frac \delta {h^{3/2}} | \pu |_{1,\ThN} \| \pv \|_{0,\ThN}.
		\end{gather*}

To bound $B_2$ we observe that using \eqref{inverseDK}
		we can write
		\begin{multline}\label{boundjumppv}
			\| \jump{\pv} \|^2_{0,\etax}
			\lesssim \frac \delta {h^2}\| \jump{\pv} \|^2_{0,K^+ \cup K^-} 
			= \frac \delta  {h^2} \| \jump{\pv - v} \|_{0,K^+ \cup K^-}\\
			\lesssim \frac \delta  {h^2} \| \pv^+ - v \|^2_{0,K^+ \cup K^-} + \frac \delta  {h^2}  \| \pv^- - v \|_{0,K^+ \cup K^-} \lesssim \delta | v |^2_{1,K^ + \cup K^-}
		\end{multline}
		Then, we can write
		\[
		\int_{\etax} \nabla \{ \pu \} \cdot \taux \jump{\pv} \lesssim  \| \nabla \{ \pu \} \cdot \taux \|_{0,\etax}  \| \jump{\pv} \|_{0,\etax} \lesssim  \frac {\delta} {h} \left(\| \nabla \pu^+  \|_{0,K^+} + \| \nabla \pu^-  \|_{0,K^-}  \right)  | v |_{1,K^+ \cup K^-} 
		\]
		whence
		\[
		B_2 \lesssim\frac {\delta} {h} | \pu |_{1,\ThN} | v |_{1,\ThN}.
		\]



		\NOTE{Poincar\'e type bounds on Dirichlet boundary elements elements
			\begin{multline*}
				\| u \|^2_{0,K} \leq \| u - \bar u^E \|^2_{0,K} + \| \bar u^E \|^2_{0,K} \lesssim h^2 | u |^2_{1,K} + | K | \, | E |^{-2} | \int_E u |^2 \\
				\lesssim h^2 | u |^2_{1,K} + | K | \, | E |^{-1}  \int_E | u |^2 \simeq h^2 | u |_{1,K}^2 + h \| u \|_{0,E}^2
			\end{multline*}
			whence
			\[
			\| u \|_{0,K} \lesssim h | u |_{1,K} + h^{1/2} \| u \|_{0,E}.
			\]
		}
	
	In order to bound $B_3$ we observe that
		\begin{multline}
			\int_{\etax} \nabla \pu^- \cdot \taux^- \pv^- \leq \| \nabla \pu^- \cdot \taux^- \|_{0,\etax} \| \pv^- \|_{0,\etax} \lesssim 
			\frac \delta {h^2}  \| \nabla \pu  \|_{0,K^-} 
			\| \pv \|_{0,K^-}\\  \lesssim 
			\frac \delta {h^2} | \pu |_{1,K^-} (h | \pv |_{1,K^-} 
			+  h^{1/2}   \| \pv \|_{0,E^-}
			),
		\end{multline}
where	$K^-  \in \ThD$ is the element on the Dirichlet side with macro vertex $\x$, and  $E^-$ is the boundary macro edge of $K^-$ that has $\x$ as a macro vertex. 		This gives 
		\[
		B_3 \lesssim  \sum_{\x \in \DNnodes} \frac \delta {h} | \pu |_{1,K^-}  \left(| \pv |_{1,K^-} 
		+   h^{-1/2}   \| \pv \|_{0,E^-}\right).
		\]
Combining all the bounds we obtain	\eqref{contpezzoNeumann}.		\end{proof}

	The following corollary deals with the case $\chi = 0$.
	\begin{corollary} For $u,v \in V_h$ it holds that
		\begin{equation}	- \int_{\GhN} \dnh{\pu} 
			\, \pv	+ \int_{\GN} \dn{\pu}  \, \pv	
			\lesssim
			\tau \Norm{ u } \Norm{ v } + \tauN  | \pu |_{1,\Th} \| \pv \|_{0,\Oh} 
			.
		\end{equation}
	\end{corollary}
	
	\begin{proof}
		We start by observing that we have the bound
		\[
		\int_{\etax}\nabla \pu^+ \cdot \etax^+ \pv^{\pm} \leq \| \nabla \pu^+ \cdot \etax^+ \|_{0,\etax} \| \pv^\pm \|_{0,\etax} \lesssim \frac \delta {h^2}
		\| \nabla \pu^+ \|_{0,K^+} \| \pv \|_{0,K^\pm}.
		\]
	Then we can write
	\begin{multline*}
		- \int_{\GhN} \dnh{\pu} 
		\, \pv	+ \int_{\GN} \dn{\pu}  \, \pv \\ 
		\lesssim 
		\tau \Norm{ u } \Norm{ v } + \max_{x \in \GhN} \tauN | \pu |_{1,\Th} \| \pv \|_{0,\Oh} + | \pu |_{1,\ThN} | v - \pv |_{1,\ThN} \\ + \Big| \sum_{\x  \in \XN} \int_{\etax}  [\nabla \pu\cdot \taux ]  \{ \pv \} \big| \\
		\lesssim 	\tau  \Norm{ u } \Norm{ v } + \tauN | \pu |_{1,\Th} \| \pv \|_{0,\Oh} + | \pu |_{1,\ThN} | v - \pv |_{1,\ThN} .
	\end{multline*}
		\end{proof}

	To conclude, we observe that we have
	\begin{multline*}
		\int_{\GhN} \dnh{\pu}  \,(v - \pv) \leq \| \dnh{\pu}  \|_{0,\GhN} \| v - \pv \|_{0,\GhN} \\
		\lesssim
		h^{1/2} \| \dnh{\pu}  \|_{0,\GhN} (h^{-1} \| v - \pv \|_{0,\ThN} +  | v - \pv |_{1,\ThN}) \lesssim | \pu |_{1,\ThN} | v - \pv |_{1,\ThN}.
	\end{multline*}
Using \eqref{condmesh} and \eqref{polypoincare} we then obtain the continuity bound \eqref{continuity}.

%
%
%
%
%
%
%
%
%
	
	\subsection{Coercivity}  
	
Let us now deal with the coercivity of the bilinear form $\ah $. 
The following lemma holds.	
\begin{lemma}\label{lem:coercivity} There exist $\beta_0$ and $\gamma_0$ such that, if $\beta > \beta_0$ and $\gamma > \gamma_0$ then the following holds: there exist $\tau_0$ such that if $\tauD<\tau_0$ and $\tauN<\tau_0$, then 
	\begin{equation}\label{coercivity}
	\ah (v,v) \gtrsim \Norm{v}^2 \qquad \forall v \in \Vh.
	\end{equation}
\end{lemma}

\begin{proof} We
observe that  Proposition \ref{inversebase} implies that
\begin{gather}
h^{-1/2}	\| \E(\pu) - \hE(\pu) \|_{0,\GhD} \lesssim \frac \delta h\, | \pu |_{1,\ThD}, \qquad h^{-1/2}	\| \pu - \hE(\pu) \|_{0,\GhD} \lesssim \frac \delta h\, | \pu |_{1,\ThD},
\end{gather}
which yields the following
 norm equivalence
\[
\Norm{u}^2\simeq | u |^2_{1,\Oh} +  h^{-1} 
\| \hE(\pu) \|^2_{0,\GhD}.
\]
It is convenient to prove coercivity with respect to the norm on the \replaced[id=mm]{right hand}{ righthand} side. We have
\begin{multline}\ah (u,u) \geq	
	\int_{\Th} | \nabla \pu |^2   
	+ \beta \alpha_* | u-\pu |^2_{1,\Th}  + \gamma h^{-1} 
	\| \hE(\pu) \|^2_{0,\GhD}  \\ \underbrace{- \int_{\GhN} \dnh{\pu}\, u 
		+ \int_{\GN} \dn{\pu}\,  \pu 
		+  \sum_{\x  \in \XN}  \int_{\etax}  [\nabla \pu\cdot \taux  ]  \{ \pu \} }_A
	\\-\underbrace{\int_{\GhD} \dnh{\pu}  \pu}_B + \underbrace{\int_{\GhD}  (\E(\pu) - \pu) \dnh{\pu}  }_C
	\\ + \underbrace{\gamma h^{-1} \int_{\GhD} (\E(\pu) - \hE(\pu)) \hE(\pu).}_D
\end{multline}


\NOTE{
\begin{multline}A = - \int_{\GhN} \dnh{\pu}\, u 
	+ \int_{\GN} \dn{\pu}\,  \pu 
	+  \sum_{\x  \in \XN}  \int_{\etax}  [\nabla \pu\cdot \taux  ]  \{ \pu \} \\=
	- \int_{\GhN} \dnh{\pu}\, \pu  - \int_{\GhN} \dnh{\pu}\, (u-\pu)
	+ \int_{\GN} \dn{\pu}\,  \pu 
	+  \sum_{\x  \in \XN}  \int_{\etax}  [\nabla \pu\cdot \taux  ]  \{ \pu \} \\
	\lesssim  | \pu |_{1,\ThN} | v - \pu |_{1,\ThN} + \tau  | \pu |^2_{1,\Th}  +  \tauN | \pu |_{1,\Th} \| \pu \|_{0,\Oh}\\
	\frac \varepsilon 2 | \pu |_{1,\ThN} + \frac 1 {2\varepsilon} | v - \pu |^2_{1,\ThN} +  \tau  | \pu |^2_{1,\Th}  + \frac \varepsilon 2 |\pu |_{1,\Th}^2 
	+ \frac {\tauN^2} {2\varepsilon} 
	\| \pu \|^2_{0,\Oh}
	\\
	\lesssim 
	\frac \varepsilon 2 | \pu |_{1,\ThN} + \frac 1 {2\varepsilon} | v - \pu |^2_{1,\ThN} +  \tau  | \pu |^2_{1,\Th}  + \frac \varepsilon 2 |\pu |_{1,\Th}^2 
	+ \frac {\tauN^2} {2\varepsilon} 
	( | \pu - u |_{1,\Th}^2 + | \pu |_{1,\Th}^2 
	+ \| \pu \|_{0,\GhD}^2 ) \\
	= \left(
	\varepsilon + \tau + \frac {\tauN^2}{2 
		\varepsilon}
	\right) | \pu |_{1,\Th}^2  + 
	\frac {1 + \tauN^2} {2\varepsilon} 
	| \pu - u |_{1,\Th}^2  + \frac{\tauN^2}{2\varepsilon}
	\| \hE(\pu) \|_{0,\GhD}^2+ \underbrace{\frac{\tauN^2}{2\varepsilon}  h 
\tau^2	| \pu \|_{1,\Th}^2 }_{\lesssim \tauN^2 | \pu |_{1,\Th}/(2\varepsilon)}
	\\
\lesssim	
\left(
\varepsilon + \tau + \frac {\tauN^2}{2 
	\varepsilon}
\right) | \pu |_{1,\Th}^2  + 
\frac {1 + \tauN^2} {2\varepsilon} 
| \pu - u |_{1,\Th}^2  + \frac{\tauN^2}{2\varepsilon}
\| \hE(\pu) \|_{0,\GhD}^2
\end{multline}
}

%
%

Combining Lemma \ref{lem:boundI} and Proposition \ref{prop:poincare} we can bound
\[| A | \leq C_A \left[\left(
\varepsilon + \tau + \frac {\tauN^2}{2 
	\varepsilon}
\right) | \pu |_{1,\Th}^2  + 
\frac {1 + \tauN^2} {2\varepsilon} 
| \pu - u |_{1,\Th}^2  + \frac{\tauN^2}{2\varepsilon}
\| \hE(\pu) \|_{0,\GhD}^2 \right]. 
\]
$| B |$, 
$| C |$ and $| D |$ can be bound as in \cite{VEM_curvo,bertoluzza2024virtual} . More precisely, we can 
bound $| B|$ as

\NOTE{
\begin{multline*}
	| B | \leq  \| \dnh \pu \|_{0,\GhD}  \| \pu \|_{0,\GhD} \lesssim \frac {\varepsilon'} 2 h \| \dnh \pu \|_{0,\GhD}^2 + \frac 1 {2\varepsilon'} h^{-1} \| \pu \|_{0,\GhD}^2\\
	 \lesssim
	\frac {\varepsilon'} 2 | \pu |_{1,\Th}^2 + \frac 1 {2\varepsilon'} h^{-1} \| \pu - \hE(\pu)\|_{0,\GhD}^2 + \frac 1 {2\varepsilon'} h^{-1}\| \hE(\pu)\|_{0,\GhD}^2\\
	\lesssim \frac {\varepsilon'} 2 | \pu |_{1,\Th}^2 +  \frac 1 {2\varepsilon'} \tauD | \pu |_{1,\Th}^2 + \frac 1 {2\varepsilon'} h^{-1} \| \hE(\pu)\|_{0,\GhD}^2 \\
	\lesssim \left(\frac {\varepsilon'} 2  +  \frac \tauD {2\varepsilon'}  \right) | \pu |_{1,\Th}^2 + \frac 1 {2\varepsilon'} h^{-1}\| \hE(\pu)\|_{0,\GhD}^2 
	\end{multline*}
}

\[
| B | \leq C_B \left[
\left(\frac {\varepsilon'} 2  +  \frac \tauD {2\varepsilon'}  \right) | \pu |_{1,\Th}^2 + \frac 1 {2\varepsilon'} h^{-1} \| \hE(\pu)\|_{0,\GhD}^2
\right].
\]

\NOTE{
\begin{equation*}
	C \leq  \| \E(\pu) - \pu  \|_{0,\GhD} \| \dnh{\pu} \|_{0,\GhD}  = h^{-1/2} \| \E(\pu) - \pu  \|_{0,\GhD}  h^{1/2}\| \dnh{\pu} \|_{0,\GhD} \lesssim
	\tau | \pu |_{1,\Th}^2	
	\end{equation*}
}

Moreover we have
\[
| C | \leq  \| \E(\pu) - \pu  \|_{0,\GhD} \| \dnh{\pu} \|_{0,\GhD} \leq 
C_C
\tauD
| \pu |_{1,\Th}^2
\]
and 
\[
| D | \leq \gamma h^{-1} \| \E(\pu) - \hE(\pu) \|_{0,\GhD} \| \hE(\pu) \|_{0,\GhD} \leq \frac{\gamma} 2  h^{-1} \| \hE(\pu) \|^2_{0,\GhD} + C_D \tauD\gamma | \pu |^2_{1,\Th}.\]

\newcommand{\dN}{\delta^N}
\newcommand{\dD}{\delta^D}

Using such bounds we obtain
\begin{multline*}
	\ah (u,u) \geq	
| \pu |_{1,\Th}^2
	+\beta \alpha_* | u-\pu |^2_{1,\Th}  + \gamma h^{-1} 
	\| \hE(\pu) \|^2_{0,\GhD}  
	\\
	- \underbrace{\left[C_A\left(
\varepsilon + \tau + \frac {\tauN^2}{2 
	\varepsilon}
\right)  + C_B 
\left( {\varepsilon}  +   \frac \tauD {4\varepsilon}  \right)	+ C_C
\tauD
 + C_D \tauD\gamma 
\right]}_{=I}| \pu |^2_{1,\Th}
\\
 - \underbrace{\left[C_A \frac{\tauN^2}{2\varepsilon}
+ C_B  \frac 1 {4\varepsilon} 
	+ \frac{\gamma} 2 \right]}_{II} h^{-1} \| \hE(\pu) \|^2_{0,\GhD} 
	- \underbrace{C_A
	\frac {1 + \tauN^2} {2\varepsilon} }_{III}
	| \pu - u |_{1,\Th}^2
\end{multline*}
We can now take $\varepsilon$ such that $(C_A + C_B) \varepsilon = 1/2$. Then we take $\gamma$ such that $C_B / (4\varepsilon) < \gamma/2$ (that is $\gamma < \gamma_0 = 4 C_B / (C_A + C_B)$). Then $\gamma - \gamma/2 - C_B/(4\varepsilon) = \gamma/2 -  C_B/(4\varepsilon) = c_1 > 0$. Provided $C_A \tauN^2 / 2\varepsilon < c_1/2$ we have that $\gamma - II > 0$. Provided $\tau < \tau_0  = (C_A + C_B/(4\varepsilon) + C_C + C_D \gamma)^{-1}/6$ and that $\tauN < \widehat{\tau}_0 = \sqrt{2\varepsilon/(6C_A)}$, we have that $I < 5/6$. Choosing $\beta > \beta_0$ with $\beta_0 = C_A (1 + \widehat{\tau}_0)/(2 \varepsilon \alpha_*)$ yields that $III < \beta \alpha_*$. 
Then \eqref{coercivity} holds with an implicit constant independent of $h$.
\end{proof}

\newcommand{\puI}{\pi_u^I}
\newcommand{\puh}{\pi_u^h}
\newcommand{\pe}{\pi_e}

\subsection{Consistency error}
We have the following Lemma

\begin{lemma}Let $u$ be the solution to problem \eqref{contpb} and assume that $u \in H^{k+1}(\O)$. Then we have that
	\[
	| \ah(u,v) - \fh(v) | \lesssim \left( A^*(h) h^k | u |_{k+1,\O} + h^{-1/2} \delta^{k+1} \| u \|_{k+1,\infty,\O} \right)  \Norm{v}.
	\]
\end{lemma}

\begin{proof}
	We recall that $\int_{\Th} \nabla \pu\cdot \nabla \pv = \int_{\Th} \nabla \pu\cdot \nabla v$. Then
	
	\begin{multline}\ah (u,v) =	\underbrace{\int_{\Th} \nabla u \cdot \nabla v - \int_{\Gh} \dnh{u} \, v }_{= -\int_{\Th} \Delta u v = \int_{\Th} f v} + \int_{\Th} \nabla (\pu - u) \cdot \nabla v +  \beta
		s_h(u-\pu,v-\pv) \\- \int_{\Gh} \dnh{(\pu - u)} \, v  + \underbrace{\int_{\GhD}  \widetilde u \Big(\dnh{\pv} + \gamma h^{-1} \hE(\pv) \Big) }_{=\int_{\GhD} \widetilde g (\dnh{\pv} + \gamma h^{-1} \hE(\pv))} +  \int_{\GhD}  (\E(\pu) - \widetilde u) \dnh{\pv}  
		\\
		+ \gamma h^{-1} \int_{\GhD} (\E(\pu) - \widetilde u) \hE (\pv)\\
    + \underbrace{\int_{\GN} \dn{u}  \pv }_{\int_{\GN} g^N \pv}
		+ \int_{\GN} \dn{(\pu - u)}  \pv 
		+ \switch \sum_{\x  \in \XN}\int_{\etax}  [\nabla( \pu - u)\cdot \taux ]  \{ \pv \}.
	\end{multline}
	Then
	\begin{multline}\label{boundABC}
    A_h(u,v) - F_h(v) =  \int_{\Oh} f v - f \Pi^0_{k-2} v +
		\underbrace{ \int_{\Th} \nabla (\pu - u) \cdot \nabla v +  \beta
			s_h(u-\pu,v-\pv)}_A \\  \underbrace{- \int_{\GhD} \dnh{(\pu - u)} \, v +  \int_{\GhD}  (\E(\pu) - \widetilde u) \dnh{\pv}  
			+ \gamma h^{-1} \int_{\GhD} (\E(\pu) - \widetilde u) \hE (\pv)}_B\\
		\underbrace{- \int_{\GhN} \dnh{(\pu - u)} \, v	+ \int_{\GN} \dn{(\pu - u)}  \pv 
			+ \switch \sum_{\x  \in \XN}\int_{\etax}  [\nabla( \pu - u)\cdot \taux ]  \{ \pv \}}_C.\\
	\end{multline}
	The terms $A$ and $B$ have been bounded in \cite{VEM_curvo}: we have
\begin{equation}\label{boundA}
A \leq A^*(h) h^k | u  |_{k+1,\O} \Norm v
\end{equation}
	and 
	\begin{equation}\label{boundB}
	| B | \lesssim  (h^k | u |_{k+1,\O}  + h^{-1/2} \delta^{k+1} \| u \|_{k+1,\infty,\Oh} )(| v |_{1,\O} + h^{-1/2}\| \pv \|_{0,\GhD}).
	\end{equation}

	Let us bound $| C |$. As in the proof of Lemma \ref{lem:boundI} we have that 
	\begin{multline}
		C = - \int_{\GhN} \dnh{(\pu - u)} \, (v - \pv)	 \\
		- \int_{\GhN} \dnh{(\pu - u)} \, \pv	+ \int_{\GN} \dn{(\pu - u)}  \pv 
		\\+  \sum_{\x  \in \XN}  \left(
		\int_{\etax} \nabla (\pu^+ - u^+)\cdot \taux^+ \pv^+ 
		+ \int_{\etax} \nabla (\pu^- - u^-) \cdot \taux^- \pv^- 
		-  \int_{\etax} \nabla \{\pu - u \} \cdot \taux \jump{\pv}  	\right) \\
		= - \int_{\GhN} \dnh{(\pu - u)} \, (v - \pv) + \sum_E 
		\left(\int_{\DK}\nabla (u - \pu) \cdot \nabla \pv + \int_{\DK} \Delta (u-\pu) \pv \right)
		\\+ \sum_{\x \in \XN} \int_{\etax }\nabla \{\pu - u \}\cdot\taux \jump{\pv} + \sum_{\x \in \DNnodes} \int_{\etax} \nabla (\pu^- - u^-) \cdot \taux^- \pv^-\\
		\lesssim
		\| \dnh{(\pu - u)} \|_{0,\Gh} \| v - \pv \|_{0,\Gh}  + \sum_E \Big(| \pu - u |_{1,\DK} | \pv |_{1,\DK} + | \pu - u |_{2,\DK} \| \pv \|_{0,\DK} \Big) \\
		+ \sum_{\x\in \XN}  \| \nabla \{ \pu - u\} \|_{0,\etax} \| \jump \pv \|_{0,\etax}  + \sum_{\x \in \DNnodes} \| \nabla(\pu^- - u^-)\cdot \etax^- \|_{0,\etax} \| \pv^- \|_{0,\etax}\\
		\lesssim 
		\| \dnh{(\pu - u)} \|_{0,\Gh} \| v - \pv \|_{0,\Gh}  + \sum_E \Big(| \pu - u |_{1,\DK} | \pv |_{1,\DK} + | \pu - u |_{2,\DK} \| \pv \|_{0,\DK} \Big)\\
		+ \sum_{\x\in \XN} \| \nabla \{\pu - u \} \|_{0,\etax} \| \jump{\pv} \|_{0,\etax} + \sum_{\x \in \DNnodes} \| \nabla (\pu-u) \|_{0,\etax} \| \pv^- \|_{0,\etax}
		\\	\lesssim 
		\Big(
		\| \nabla (\pu - u) \|_{0,\ThN} + h \| \nabla (\pu - u) \|_{1,\ThN}
		\Big) | v - \pv |_{1,\ThN}
		\\
		+  \sqrt{\frac \delta h} \sum_{K\in \ThN} \Big(| \pu - u |_{1,\Kt} | \pv |_{1,\KE } + | \pu - u |_{2,\Kt} \| \pv \|_{0,K} \Big)\\
		+ \Big(h^{-1/2}\| \nabla (\pu-u) \|_{0,\ThN} + h^{1/2}\| \nabla (\pu-u) \|_{1,\ThN}\Big) \sqrt{\delta} | v |_{1,\ThN}  \\
		+ \sqrt{\frac\delta h } \sum_{\x \in \DNnodes}    \Big(
		| \pu - u |_{1,K_\x^-} + h | \nabla (\pu - u ) |_{1,K_\x^-}   \Big)
		\Big(
		| \pv |_{1,K_\x^-}  + h^{-1/2} \| \pv\|_{0,E_\x^-}
		\lesssim 
		h^k | u |_{k+1,\O} \Norm{ \pv }.
	\end{multline}
	That is
	\begin{equation}\label{boundC}
	| C  | \lesssim h^k | u |_{k+1,\O} \Norm{ \pv } \lesssim h^k | u |_{k+1,\O} \Norm{ v }.
	\end{equation}
	Using \eqref{boundA}, \eqref{boundB} and \eqref{boundC} in \eqref{boundABC} yields the desired result.
\end{proof}

\NOTE{
	Proof of the corollary. We have
	\begin{multline*}
		\Norm{u_h - v_h} \lesssim \ah(u_h - v_h,u_h - v_h)
		= \ah(u_h - u,u_h - v_h) + a(u - v_h,u_h - v_h) \\=
		\fh(u_h - v_h) - \ah(u,u_h - v_h) + \Norm{u - v_h} \Norm{u_h - v_h}	\end{multline*}
	
}

\NOTE{
	We have (see \eqref{boundjumppv})
	\[	\| \jump{\pv} \|_{0,\etax}
	\lesssim \sqrt{\delta} | v |_{1,K^ + \cup K^-}\]
	Moreover, recalling that for $\x \in \DNnodes$, $E^- \in \partial K^-$ is on the Dirichlet side, we have
	\begin{multline*}
		\| \nabla(\pu^--u^-) \|_{0,\etax} \| \pv \|_{0,\etax} \lesssim
		\| \nabla(\pu^--u^-) \|_{0,\partial \Kt^-} \sqrt{\frac \delta {h^2} } \| \pv \|_{0,K^-}\\
		\lesssim  \sqrt{\frac\delta h }     \Big(
		| \pu^- - u^- |_{1,K^-} + h | \nabla (\pu^- - u^- ) |_{1,K^-}   \Big)
		\Big(
		| \pv |_{1,K^-}  + h^{-1/2} \| \pu \|_{0,E^-}
		\Big)
	\end{multline*}
	
}

\NOTE{
	I look at $\PinablaK$ as  $\PinablaK : H^1(\Kt) \to \Poly{k}(\tK)$.  I have
	\[
	\| \PinablaK(u) \|_{1,\Kt} \lesssim \| \PinablaK(u) \|_{1,K} \lesssim \| u \|_{1,K} \lesssim \| u \|_{1,\Kt}.
	\]
	Then $\PinablaK$ is bounded from $H^1(\Kt)$ to $H^1(\Kt)$ and preserves polynomials. The
	\[
	\| u - \PinablaK u \|_{1,\Kt} = \| u - p - \PinablaK(u - p) \|_{1,\Kt} \lesssim \| u - p \|_{1,\Kt} \quad \Longrightarrow 
	\quad		\| u - \PinablaK u \|_{1,\Kt} \lesssim \inf_{p\in \Poly{k}} \| u - p \|_{1,\Kt}.
	\]
	
	\begin{multline*}
		\|  \nabla(u - \PinablaK u) \|_{1,\Kt}  = 
		\| \nabla  u - \mathbf{p} \|_{1,\Kt} + \| \mathbf{p}- \nabla \PinablaK u \|_{1,\Kt}
		\\[4mm]
		\lesssim h^{k-1} | u |_{k+1,\Kt} + h^{-1}\| \mathbf{p}- \nabla \PinablaK u \|_{0,K}	\\[4mm]
		\lesssim h^{k-1} | u |_{k+1,\Kt} + h^{-1}\| \mathbf{p} - \nabla  u \ \|_{0,K}	+ h^{-1}\|  \nabla(u - \PinablaK u)  \|_{0,K} 
		\\[4mm]
		\lesssim h^{k-1} | u |_{k+1,\Kt} + h^{-1}\| \mathbf{p} - \nabla  u \ \|_{0,\Kt}	+ h^{-1}\|  \nabla(u - \PinablaK u)  \|_{0,K} \lesssim h^{k-1} | u |_{k+1,\Kt}.
	\end{multline*}
	where I used Aubin-Nitsche on $\Kt$.Everything should work, and yield
	\[
	\|  \nabla(u - \PinablaK u) \|_{1,\Kt}  \lesssim h^{k-1} | u |_{k+1,\Kt}
	\]
	
}

%
%
%
%
%
%
%
%
%

\section{Numerical tests}\label{sec:numerical}
 
We tested the proposed method on problem \eqref{contpb}, with $\Omega$ being the circle of center $(0.5,0.5)$ and radius $0.5$. The boundary $\partial \Omega$ is split as $\partial \Omega = \bar\Gamma^D \cup \bar \Gamma^N$ with $\GD = \{ (x,y) \in \partial \Omega: y < 0.5\}$ and $\GN = \{ (x,y) \in \partial \Omega: y > 0.5\}$. We take $f$, $g^D$ and $g^N$ in such a way that the solution $u$ is the Franke function \cite{franke1979critical}
\begin{multline}
	u(x,y):=\frac 3 4 e^{-\left(
		(9x-2)^2+(9y-2)^2
		\right)/4} + \frac 3 4 e^{-
		\left(
		(9x+1)^2/49 + (9y+1)/10\right)}\\
	+ \frac 1 2 e^{-
		\left(
		(9x-7)^2 + (9y-3)^2	
		\right)/4
	}
	+\frac 1 5 e^{-
		\left(
		(9x-4)^2 + (9y-7)^2
		\right)	
	} . \label{franke}
\end{multline} 
We tested the method proposed with $\delN = \delD = h / 2^{n_\text{ref}}$, for different values of $n_\text{ref}$, $\beta$ and $\gamma$.
\added[id=mm]{We used the orthonormal polynomial basis described in \cite{Mascotto_illcond}.}
Letting $u_h$ denote the discrete solution obtained by the order $\p$ VEM method proposed in the previous section, for all the tests we consider the  relative error in the $H^1(\Th)$ seminorm, as well as in the $L^2(\Oh)$ norm. For the VEM case, these are, as usual, approximated as
\begin{align}\label{eS}
	e^u_1 &:= \frac{
		\| \nabla u - \Pi^0_{k-1}(\nabla u_h) \|_{0,\Oh}
	}{| u |_{1,\Oh}}, & e^u_0 &:= \frac{\| u - \Pi^0_{k}  u_h\|_{0,\Oh}}{\| u \|_{0,\Oh}}.
\end{align}
For both test cases, we consider values of $k$ between $1$ and $6$.

\subsection{DOFs vs Mesh Size}
\label{sec:dofs}
We start by showing, 
in Figure~\ref{fig:dofs-per-k}, the relationship between the number of active degrees of
freedom (DOFs) and the mesh size $h$ for different polynomial order $k$ and different refinement levels $n_\text{ref}$.  As we can see, we consistently have $\text{DOFs} \sim h^{-2}$ and increasing the refinement level (or equivalently, decreasing $\delta$) does not imply an increase of the number of degrees of freedom.

\begin{figure}[htbp]
	\centering
	\includegraphics[width=\textwidth]{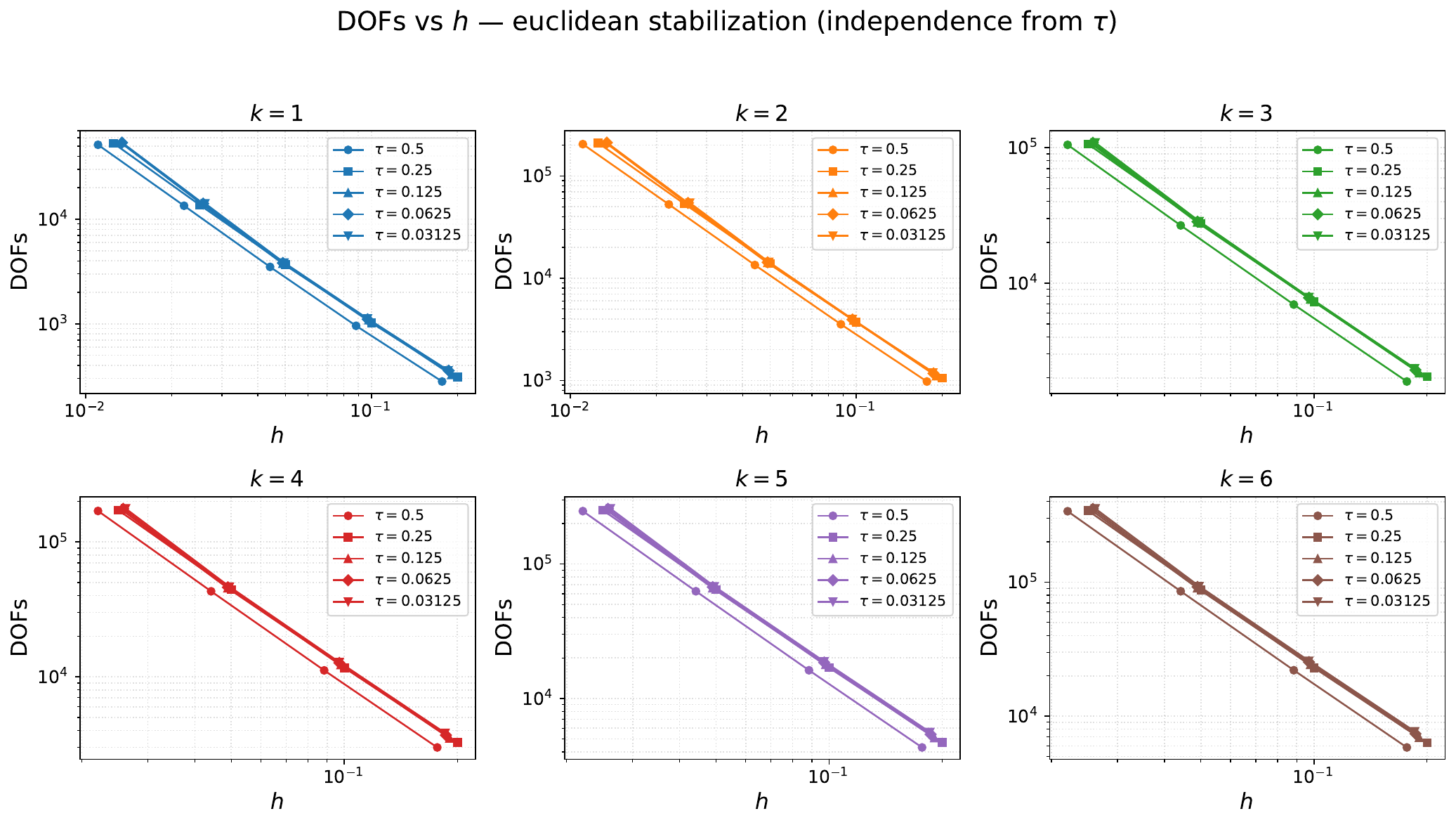}
	\caption{DOFs vs $h$ for each $k$, with different $\tau$ values
		overlaid (euclidean stabilization, $\gamma=100$, $\beta=1$).
		The curves for different $\tau$ collapse, confirming that the
		number of DOFs is independent of $\tau$.}
	\label{fig:dofs-per-k}
\end{figure}

\subsection{Convergence for fixed $n_\text{ref}$}
In Figures~\ref{fig:H1-comparison-stab} and~\ref{fig:L2-comparison-stab}, we plot the $H^1$ and $L^2$ errors as a function of the meshsize $h$ for different walues of $\tau = \delta/h$. We display the results for two values of the VEM stabilization parameter $\beta$: $\beta = 1$ and $\beta = 1000$.  In both cases two stabilizations strategies are considered: the standard {\em dofi-dofi} (euclidean) stabilization \added[id=mm]{\cite{basicVEM}} and the new robust stabilization proposed in Section \ref{sec:robust-stabilization}. The numerical results are in agreement with the theoretical ones. If $\tau$ is not small enough we observe a degradation in the convergence of the method. We also observe that, for $\beta = 1$, the two stabilization strategies yield fundamentally similar results.

\begin{figure}[htbp]
	\centering
	\begin{subfigure}[b]{\textwidth}
		\includegraphics[width=\textwidth]{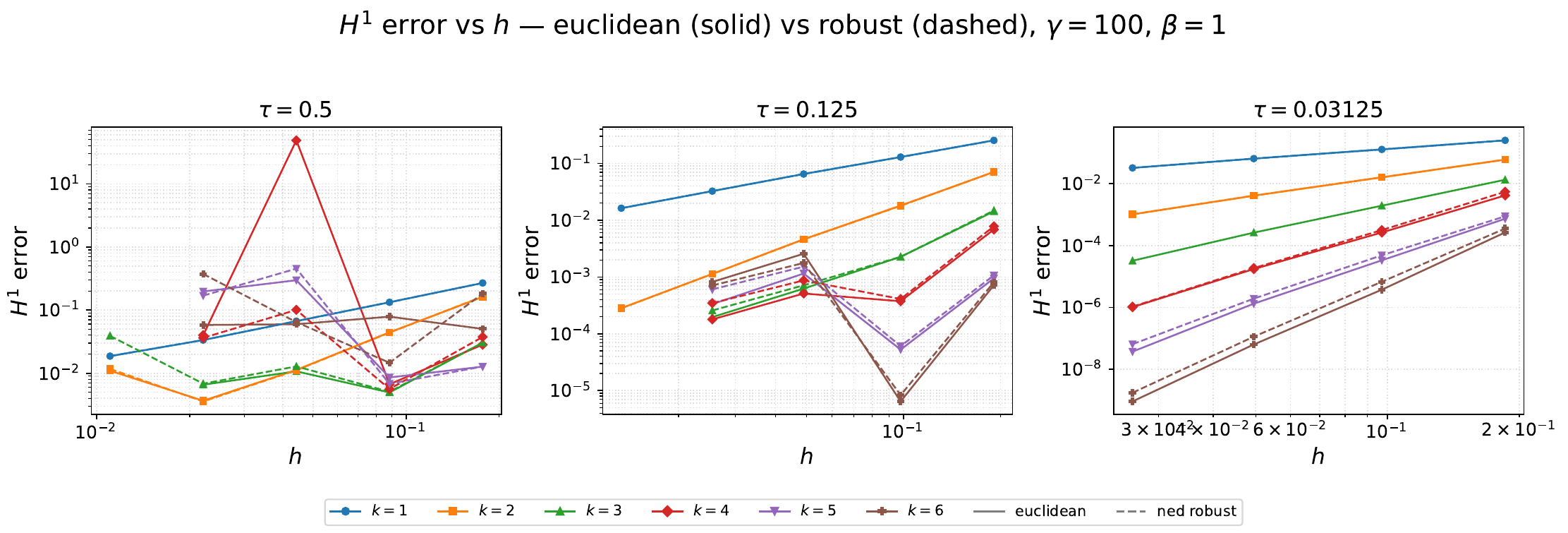}
	\end{subfigure}\\[6pt]
	\begin{subfigure}[b]{\textwidth}
		\includegraphics[width=\textwidth]{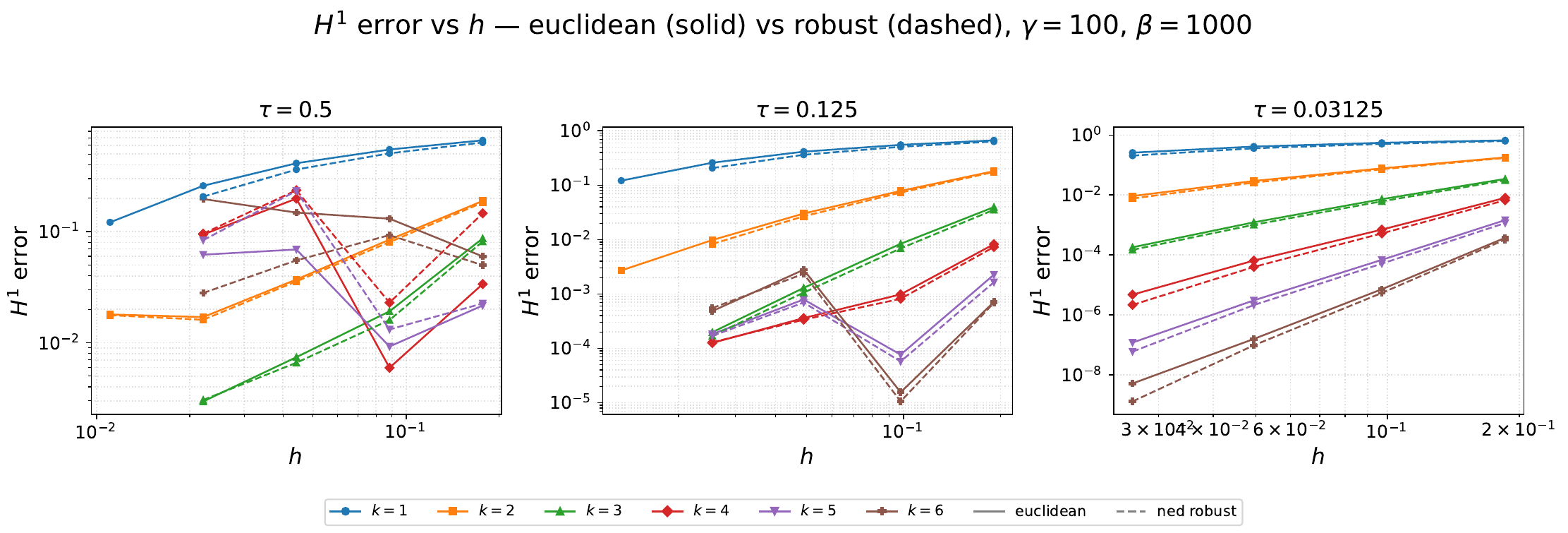}
	\end{subfigure}
	\caption{$H^1$ error: euclidean (solid) vs robust (dashed), subplots
		by $\tau$.  The advantage of the robust stabilization is clearly
    visible in the small-$\tau$ panels for $\beta=1000$ \added[id=dp]{as $h$ decreases}.}
	\label{fig:H1-comparison-stab}
\end{figure}

\begin{figure}[htbp]
	\centering
	\begin{subfigure}[b]{\textwidth}
		\includegraphics[width=\textwidth]{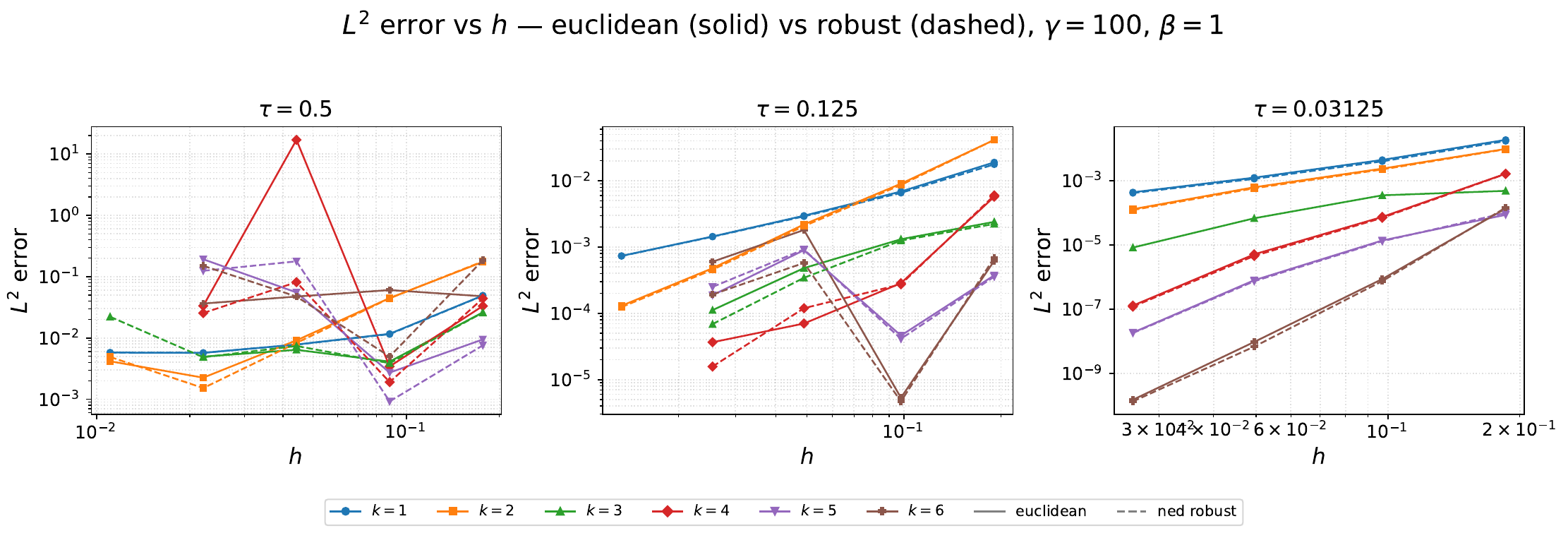}
	\end{subfigure}\\[6pt]
	\begin{subfigure}[b]{\textwidth}
		\includegraphics[width=\textwidth]{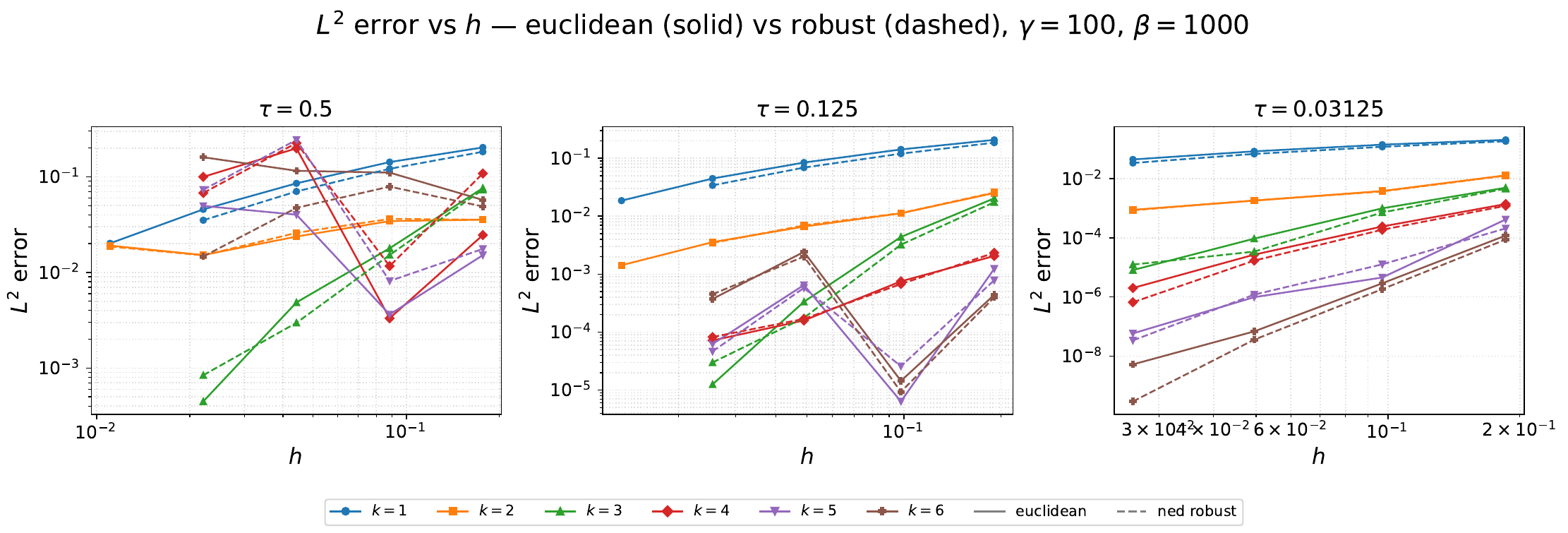}
	\end{subfigure}
	\caption{{$L^2$ error: euclidean (solid) vs robust (dashed), subplots
		by $\tau$.  For  $\beta=1000$, we generally observe a higher stability of the robust stabilization.}}
	\label{fig:L2-comparison-stab}
\end{figure}

\subsection{Convergence for optimal choice of $\tau$}

In Figure~\ref{fig:H1-ovl-sig05}, we finally plot the results obtained by choosing a sequence of meshes satisfying a condition of the form \eqref{condmesh}. To this aim, as $h$ decreases we  choose elements of diameter $H$ which are agglomerates of squared elements of meshes of mesh size $h$ with $h = H / 2^{N_\text{ref}}$, $N_\text{ref}$ chosen in such a way that
\[
\tau = \frac{h}{H^{3/2}} = \frac{2^{-N_\text{ref}}}{H^{1/2}} \leq \sigma.
\]
Once again we report the resulting convergence plots for $k =1,\dots,6$, and two values of the stabilization parameter $\beta$ ($\beta = 1$ and $\beta=1000$), with $\sigma=0.5$. While the higher stability obtained for the robust stabilization \eqref{stabneumann} is apparent for higher polynomial degrees and $\beta = 1000$, for $\beta = 1$ the two stabilizations essentially behave comparably, making the simpler euclidean stabilization preferable.
\begin{figure}[htbp]
	\centering
	\begin{subfigure}[b]{\textwidth}
		\includegraphics[width=\textwidth]{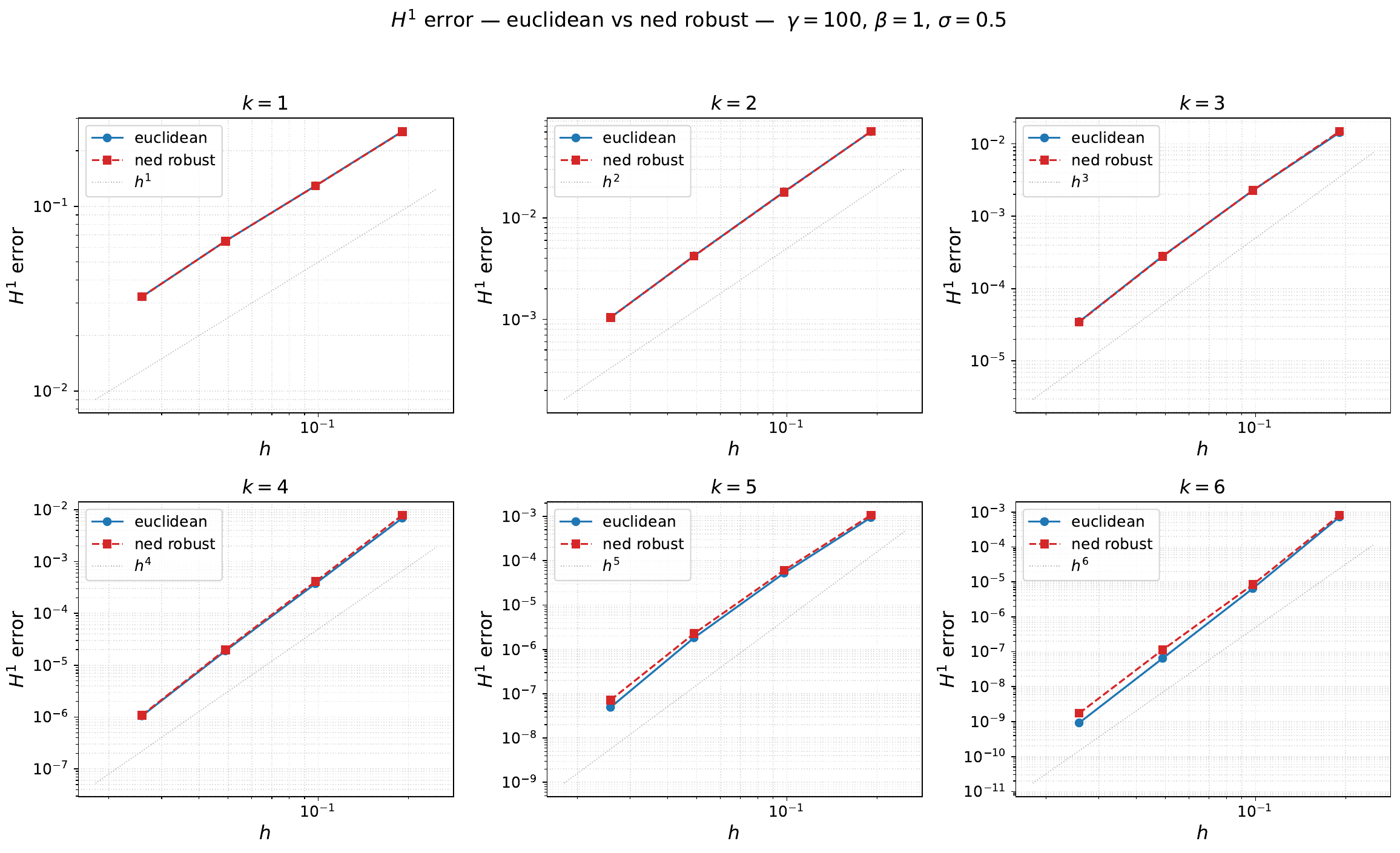}
	\end{subfigure}\\[6pt]
	\begin{subfigure}[b]{\textwidth}
		\includegraphics[width=\textwidth]{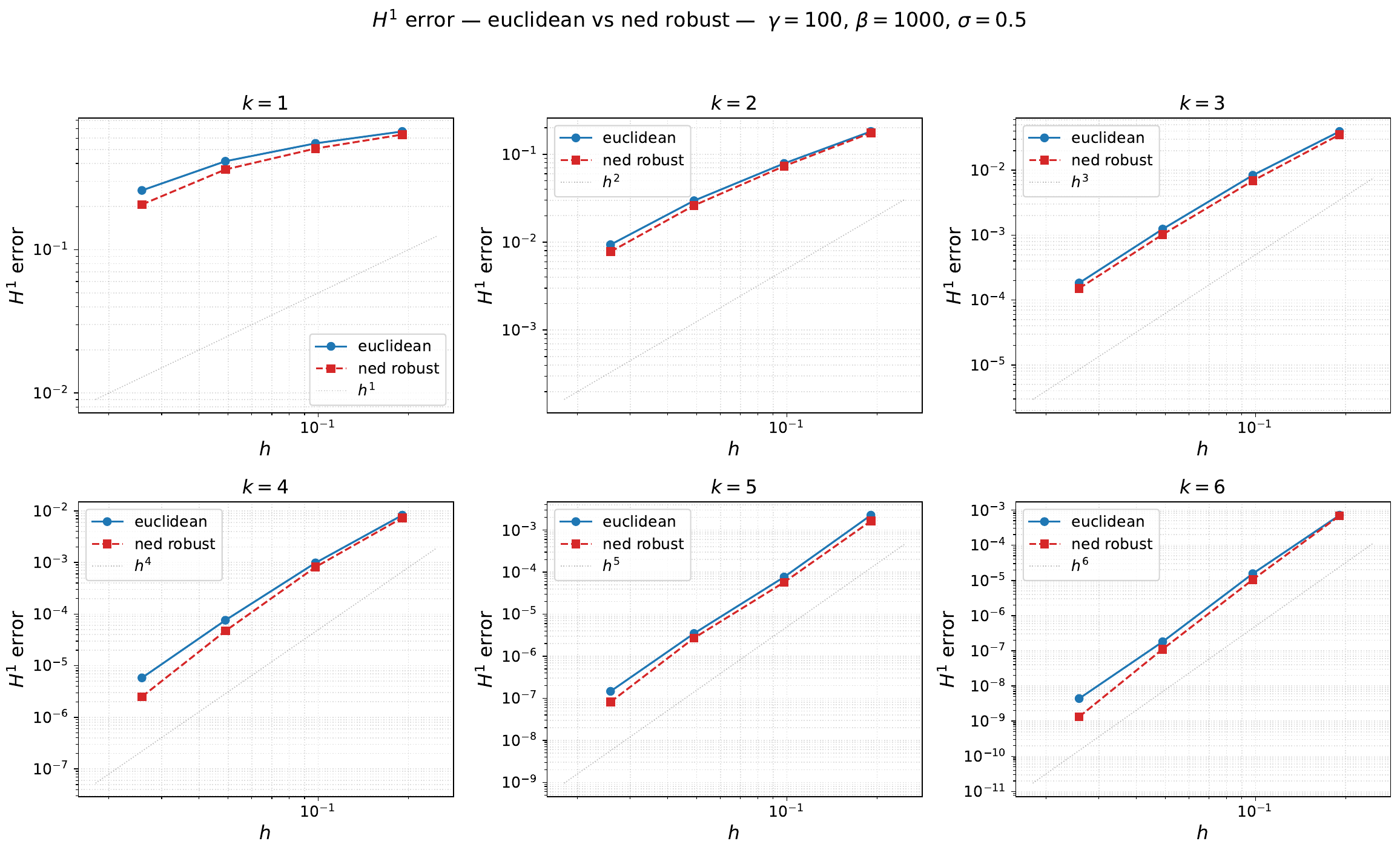}
	\end{subfigure}
	\caption{$H^1$ error, $\sigma=0.5$ — overlay:
		euclidean (blue, solid) vs ned robust (red, dashed).}
	\label{fig:H1-ovl-sig05}
\end{figure}

\NOTE{Le figure di Fig \ref{fig:H1-ovl-sig05} sono generate col codice /Users/bertoluzza/Dropbo
	x/ARTICOLI/TrabecularBone/DominiSquadrettati/Neumann/Figure/generate\_plots\_stab\_comparison.py}

\subsubsection{Condition numbers}
We conclude our experiments by comparing the condition numbers of the stiffness matrices arising in our method. Here we clearly see the lack of stability for higher order method when $\tau = h/H$ is too large. For $k=6$ and $\beta = 1$ the condition number for $\tau=.5$ is of the order of $10^{18}$ while it decreases to $10^9\sim 10^{10}$ for $\tau = 0.03125$. A similar behaviour can be observed for $\beta = 1000$ where, however, the superiority of the robust stabilization \eqref{stabneumann} can be observed. For $k=6$, both stabilizations result in a condition number of the order of $10^{18}$ for $\tau = 0.5$, while for $\tau = 0.03125$ the condition number for the robust stabilization is around $10^{11}$, one order of magnitude lower than the condition number for the euclidean stabilization.
\begin{figure}[htbp]
	\centering
	\begin{subfigure}[b]{\textwidth}
		\includegraphics[width=\textwidth]{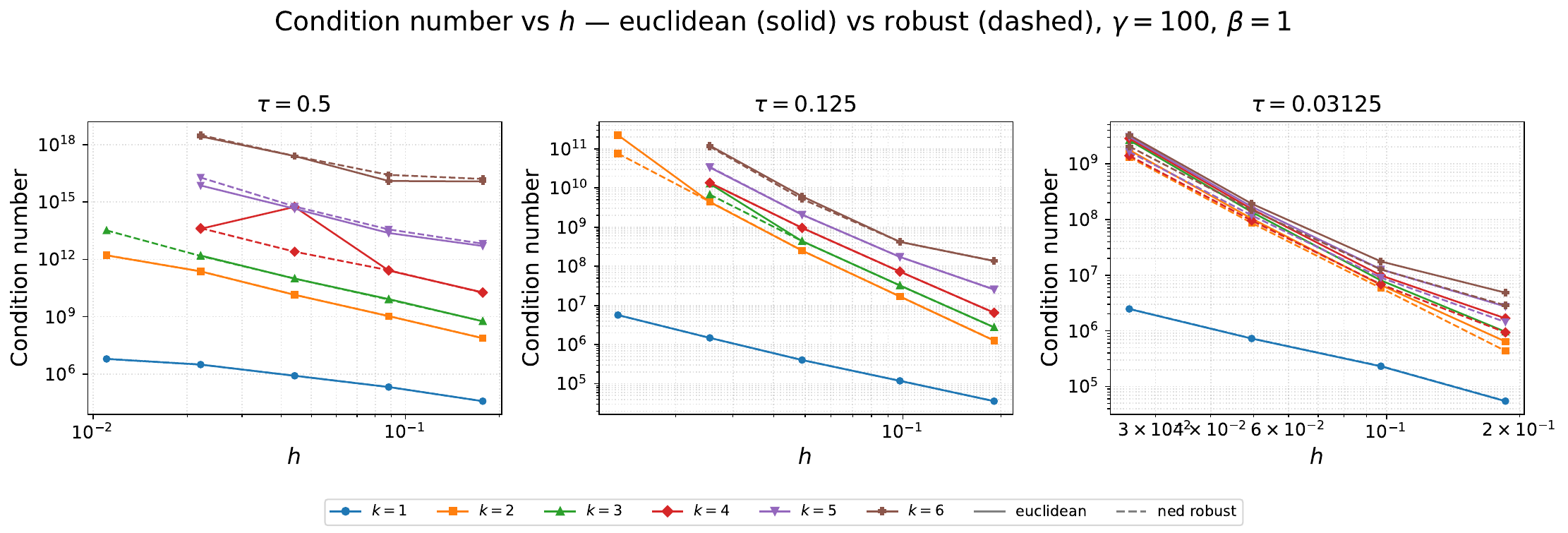}
	\end{subfigure}\\[6pt]
	\begin{subfigure}[b]{\textwidth}
		\includegraphics[width=\textwidth]{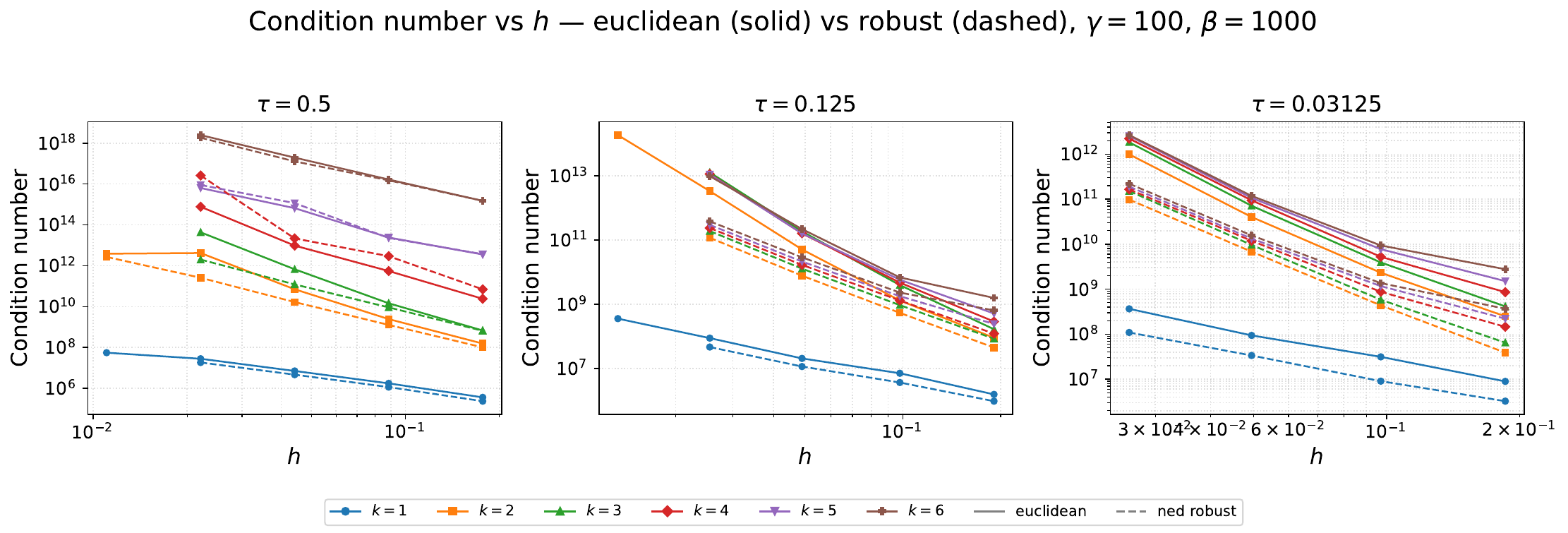}
	\end{subfigure}
	\caption{Condition number: euclidean (solid) vs robust (dashed),
		subplots by $\tau$.  This directly shows the conditioning advantage
		of the robust stabilization, especially for large values of $\beta$.}
	\label{fig:cond-comparison-stab}
\end{figure}

\bibliographystyle{amsplain}

\bibliography{biblio}

\appendix
\section{Proof of Lemma \ref{embedding}}\label{sec:embedding}
We recall that we have
\[
\| f \|_{H^s_0(\domain)}^2 \lesssim  \| f \|^2_{H^{s}(\domain)}  +   \frac 1 s \int_\domain \,d\sigma \frac{| f(\sigma) |^2}{\dist(\sigma,\partial\domain)^{2s}}, \]
where the constant in the inequality is independent of $\domain$. 
\NOTE{It is easy to see that the constant in the above inequality is independent of $\domain$ and on $s$: indeed, we have that $s > 0$ $\lra$ $a^s > 1$ ($a > 1$) $\lra$ $(1/a)^s = 1/a^s < 1$ 
	\begin{multline*}
		\int_\domain \,d\sigma \int_{\mathbb{R}^2\setminus\domain} \,d\tau \frac{| f (\sigma) |^2} {| \sigma - \tau|^{2s+2}} \leq 
		\int_\domain \,d\sigma | f (\sigma) |^2 \int_{\mathbb{R}^2\setminus B(\sigma,\dist(\sigma,\partial\domain))} \,d\tau \frac{1} {| \sigma - \tau|^{2s+2}}		\\
		= 	\int_\domain \,d\sigma | f (\sigma) |^2 \int_{\mathbb{R}^2\setminus B(0,\dist(\sigma,\partial\domain))} \,d\tau \frac{1} {|  \tau|^{2s+2}}	\\
		= 	\int_\domain \,d\sigma | f (\sigma) |^2 \int_{\dist(\sigma,\partial\domain)}^{\infty}\,d r \int_0^{2\pi} r d\theta r^{-(2s+2)}\\
		= 	\int_\domain \,d\sigma | f (\sigma) |^2 \int_0^{2\pi}  d\theta  \int_{\dist(\sigma,\partial\domain)}^{\infty}\,d r r^{-2s-1}
		\\	= \frac{\pi}s 	\int_\domain \,d\sigma | f (\sigma) |^2 \dist(\sigma,\partial\domain)^{-2s}.
\end{multline*}}
%
%
%

\NOTE{We have
	\[
	\lim_{s\to 0} \frac{a^s}{s} =
	\lim_{s\to 0} \frac {(e^{\log a})^s}s = \lim_{s\to 0} \frac {(e^{s \log a})}s 
	\]}

We then need to bound the second term on the  right hand side. We easily see that
\[
\int_\domain \,d\sigma \frac{| f(\sigma) |^2}{\dist(\sigma,\partial\domain)^{2s}} \leq \| f \|^2_{L^\infty(\domain)} C(s,\domain)
\]
with 
\[
C(s,\domain) =	\int_\domain \,d\sigma \frac{1}{\dist(\sigma,\partial\domain)^{2s}}
\]
being a constant depending on $s$ and on $\domain$. To bound such a constant, we introduce a Besicovitch cover of $\domain$ defined as follows
\[
\Besicovich = \{B(\sigma,r(\sigma)), \sigma \in \domain, \text{ with } r(\sigma) = \dist(\sigma,\partial\domain)/2\},
\]
where $B(\sigma,r(\sigma))$ is the ball centered at $\sigma$ with radius $r(\sigma)$. By the Besicovitch covering theorem there exist countable subsets $\Besk$, $k = 1, \dots, N$, $N$ being a constant only depending on the space dimension (which in our case is $2$), such that for all $k$ the elements of $\Besk$ are disjoint and 
\(
\domain \subseteq \bigcup_k \bigcup_{B\in \Besk} B
\). This implies that
\[
\one_\domain \leq \sum_k 
\sum_{B\in \Besk} \one_B.
\]
Then we can write
\[
C(s,\domain) \leq \sum_k \sum_{B\in \Besk}	\int_B \,d\sigma \frac{1}{\dist(\sigma,\partial\domain)^{2s}}.
\]
We now observe that, by construction of $\Besicovich$, for all $\sigma \in B$,  $B \in \Besicovich$ we have that
\[
\dist(\sigma,\partial\domain) \simeq r_B.
\]
Then we can write
\[
C(s,\domain) \lesssim \sum_k \sum_{B\in \Besk}	\int_B \,d\sigma \frac{1}{r_B^{2s}} \lesssim  \sum_k \sum_{B\in \Besk}	r_B^{2-2s}.
\]
Let us now split the family $\Besk$ as $\Besk = \bigcup_{j=j_0} \Beskj$ with 
\[
  \Beskj = \{ B \in \Besk: \  2^{-j-1} \leq r_B < 2^{-j}\},
\]
$j_0$ being the first index for which $\bigcup_k \Beskj \ne\emptyset$. We see that $j_0 \simeq |\log_2(\rho_G)|$.
We have
\[
\sum_{B\in \Besk}	r_B^{2-2s} \lesssim  \sum_{j=j_0 }^\infty\#(\Beskj) 2^{-j(2-2s)}.
\]
To estimate the cardinality of $\Beskj$ we observe that all the (disjoint) balls in $\Beskj$ are contained in a strip $S_j$ of thickness $\sim 2^{-j}$ adjacent to the boundary of $\domain$. We see that $| S_j | \simeq | \partial\domain| 2^{-j}$. Then
\[
  \#(\Beskj) 2^{\added[id=dp]{-}2j} \lesssim | \bigcup_{B \in \Beskj}B | \leq | S_j | \lesssim  | \partial\domain| 2^{-j}.
\]
which yields
\[
\#(\Beskj) \lesssim  | \partial\domain|  2^j.
\]
Then
\[
C(s,\domain) \lesssim N  | \partial\domain|  \sum_{j=j_0}^
\infty 2^j 2^{-j(2-2s)} \lesssim  | \partial\domain | 2^{-j_0(1-2s)} \sum_{j=0}^\infty 2^{-(1-2s)j} \lesssim  | \partial\domain |  \rho_\domain^{1-2s} \frac 1 {1-2s},
\]
and
\[
\| f \|_{H^s_0(\domain)}^2 \lesssim  \| f \|^2_{H^s(\domain)}  + 
| \partial\domain |  \rho_\domain^{1-2s} \frac 1 {s(1-2s)}  \| f \|^2_{L^\infty(\domain)}.  \]

\NOTE{
	\[
	\|p \|_{H^{1/2-\varepsilon}_0(\DK)}^2 \lesssim  | p |^2_{H^{1/2-\varepsilon}(\DK)}  + 
	| \partial\domain |  \rho_\domain^{2\varepsilon} \frac 1 {1-2s}  \| p \|^2_{L^\infty(\DK)}  
	\leq \frac \delta h | p |^2_{H^{1/2-\varepsilon}(K)}  + h \delta^{1 - 2\varepsilon} \frac 1 {\varepsilon} \| p \|^2_{L^\infty(K)}
	\]
}


\end{document}